\documentclass{amsart}
\usepackage{latexsym,amsxtra,amscd,ifthen}
\usepackage{amsfonts}
\usepackage{verbatim}
\usepackage{amsmath}
\usepackage{amsthm}
\usepackage{amssymb}
\usepackage{url}
\usepackage[arrow,matrix]{xy}
\usepackage{enumitem}

\allowdisplaybreaks[4]

\newtheorem{theorem}{Theorem}[section]

\newtheorem{corollary}[theorem]{Corollary}

\newtheorem{convention}[theorem]{Convention}
\newtheorem{lemma}[theorem]{Lemma}

\newtheorem{notation}[theorem]{Notation}

\newtheorem{proposition}[theorem]{Proposition}
\newtheorem{prop-def}[theorem]{Proposition-Definition}
\theoremstyle{definition}
\newtheorem{definition}[theorem]{Definition}
\newtheorem{remark}[theorem]{Remark}
\newtheorem{question}[theorem]{Question}

\newtheorem{remarks}[theorem]{Remarks}

\numberwithin{equation}{section}

\def\deg{{\rm deg}}

\def\Ext{{\rm Ext}}
\def\type{{\rm type}}
\def\End{{\rm End}}
\def\Tor{{\rm Tor}}

\def\Hom{{\rm Hom}}

\def\1{\mathbbold{1}}

\def\im{{\rm im}}

\def\GKdim{{\rm GKdim}\,}
\def\PIdeg{{\rm PIdeg}\,}
\def\gldim{{\rm gldim}\,}
\def\projdim{{\rm projdim}\,}
\newcommand{\PP}{\mathfrak P}
\newcommand*{\inth}{\textstyle \int}

\begin{document}
\title[IHOEs in positive characteristic]
{Iterated Hopf Ore Extensions in positive characteristic}

\author{K.A. Brown}
\address{School of Mathematics and Statistics\\
University of Glasgow\\ Glasgow G12 8QQ\\
Scotland}
\email{ken.brown@glasgow.ac.uk}

\author{J.J. Zhang}
\address{Department of Mathematics\\
Box 354350\\
University of Washington\\
Seattle, WA 98195, USA}
\email{zhang@math.washington.edu}

%\thanks{}
\subjclass[2010]{16T05;16S36;16R50}

%16T05  View Publications (2010-now) Hopf algebras and their applications 
%16S36  View Publications (1991-now) Ordinary and skew polynomial rings and semigroup rings
%16R50  View Publications (1991-now) Other kinds of identities 
%       (generalized polynomial, rational, involution)

\keywords{Hopf Ore extension; Ore extension; Hopf algebra; 
PBW deformation; Nakayama automorphism; positive characteristic;
polynomial identity}

\begin{abstract}
Iterated Hopf Ore extensions (IHOEs) over an algebraically closed 
base field $\Bbbk$ of positive characteristic $p$ are studied. We
show that every IHOE over $\Bbbk$ satisfies a polynomial identity, 
with PI-degree a power of $p$, and that it is a filtered 
deformation of a commutative polynomial ring. We classify all 
$2$-step IHOEs over $\Bbbk$, thus generalising the classification 
of 2-dimensional connected unipotent algebraic groups over $\Bbbk$. 
Further properties of $2$-step IHOEs are described: for example 
their simple modules are classified, and every 2-step IHOE is shown 
to possess a large Hopf center and hence an analog of the restricted 
enveloping algebra of a Lie $\Bbbk$-algebra. As one of a number of 
questions listed, we propose that such a restricted Hopf algebra may 
exist for every IHOE over $\Bbbk$.
\end{abstract}

\maketitle

\dedicatory{}%
\commby{}%
%\begin{center}
%\end{center}

\setcounter{section}{-1}
\section{Introduction}
\label{xxsec0}

\subsection{}
\label{xxsec0.1} 
Let $p$ be a prime and let $\Bbbk$ be an algebraically 
closed field of characteristic $p$. A famous result of 
Jacobson \cite{Ja1} states that the enveloping algebra 
$U(\mathfrak{g})$ of a finite dimensional Lie algebra 
$\mathfrak{g}$ over $\Bbbk$ is a finitely generated 
module over its center, denoted by $Z(\mathfrak{g})$. 
In fact it is easy to see from Jacobson's proof that the 
\emph{PI-degree} $d$ of $U(\mathfrak{g})$ is a power 
of $p$, (the PI-degree being by definition the square 
root of $\dim_{Q(Z(\mathfrak{g}))}Q(U(\mathfrak{g}))$, 
where $Q(U(\mathfrak{g}))$ is the central simple 
quotient division ring of $U(\mathfrak{g}))$.

Our first aim is to extend Jacobson's result to certain 
connected Hopf $\Bbbk$-algebras, as follows. Hopf Ore 
extensions $A[x;\sigma, \delta]$ were defined and studied 
by Panov \cite{Pa} in 2003. His definition was refined 
and extended in \cite{BOZZ} and then in \cite{Hua}, to 
\emph{iterated Hopf Ore extensions} (IHOEs) over a field $F$. 
(The field $F$ may have characteristic zero.)
An $n$-step IHOE over $F$ is a Hopf $F$-algebra $H$ with 
a finite chain of Hopf subalgebras 
$$F = H_0 \subset \cdots \subset H_n = H,$$ 
such that, for $i=1,\ldots,n$, $H_i = H_{i-1}[x_i;\sigma_i,\delta_i]$, 
for an algebra automorphism $\sigma_i$ and a 
$\sigma_i$-derivation $\delta_i$. See Definition \ref{xxdef1.1} 
for details, and recall that every IHOE is \emph{connected}, 
meaning that its coradical is $F$, by  
\cite[Proposition 2.5]{BOZZ}. Our first main result is

\begin{theorem}
\label{xxthm0.1}
Let $p$ and $\Bbbk$ be as above, and let $H$ be an $n$-step 
$\Bbbk$-IHOE.
\begin{enumerate}
\item[(1)] 
{\rm (Corollary  \ref{xxcor2.4})}
$H$ is a finitely generated module over its center 
$Z(H)$, which is a finitely generated normal $\Bbbk$-algebra. 
\item[(2)] 
{\rm (Theorem \ref{xxthm3.3})} 
The PI-degree of $H$ is a power of $p$.
\item[(3)] 
{\rm (Theorem \ref{xxthm3.8})}
The order of the antipode of $H$ divides $4p^{n-1}$, 
and the order of the Nakayama automorphism of 
$H$ divides $2p^n$.
\end{enumerate}
\end{theorem}

The proofs of (1) and (2) of the theorem in Sections 
\ref{xxsec2} and \ref{xxsec3} make use of results of 
\cite{LeM, CL1, CL2} on Ore extensions satisfying a 
polynomial identity (PI), and ultimately hinge on 
fundamental results of Kharchenko \cite{Kh}; some 
aspects of the argument may have relevance for the 
study of general Ore extensions, that is, beyond the 
realm of Hopf algebras.

A number of important homological consequences follow 
by routine arguments from Theorem \ref{xxthm0.1}, with 
details included in Corollary \ref{xxcor2.4} - namely, 
$H$ has global and Gel'fand-Kirillov dimensions equal 
to $n$, and is a homologically homogeneous maximal order, 
and hence is skew Calabi-Yau. 

\subsection{}
\label{xxsec0.2} 
By contrast with Theorem \ref{xxthm0.1}(1), an 
IHOE over a field $F$ of characteristic 0 does not 
satisfy a polynomial identity unless it is 
commutative, \cite[Theorem 5.4]{BOZZ}. Our second 
main result, however, generalises an aspect of Lie 
theory where there is no such dichotomy between 
characteristic 0 and positive characteristic: namely, 
it is of course fundamental to the study of 
enveloping algebras that if $\mathfrak{g}$ is a 
Lie algebra of finite dimension $n$ over a field $F$, 
then $U(\mathfrak{g})$ has an ascending filtration 
whose associated graded algebra is the 
commutative polynomial algebra in $n$ variables over 
$F$. Zhuang proved in \cite[Theorem 6.10]{Zh} that 
the same conclusion holds for a connected Hopf 
$F$-algebra of finite Gel'fand-Kirillov dimension 
$n$, when $F$ is algebraically closed of 
characteristic 0. Zhuang's result does not extend 
to positive characteristic - for example, the group 
algebra over $\Bbbk$ of the cyclic group of order 
$p$ is a connected Hopf algebra. Nevertheless, we 
prove here:

\begin{theorem}[Theorem \ref{xxthm4.3}]
\label{xxthm0.2}  
Let $p$ and $\Bbbk$ be as in Theorem \ref{xxthm0.1},
and let $H$ be an $n$-step $\Bbbk$-IHOE. 
Then positive integer degrees can be assigned to 
the defining variables of $H$ so that the corresponding 
filtration has associated graded algebra which is the 
commutative polynomial $\Bbbk$-algebra on $n$ variables.
\end{theorem} 

One can succinctly state this result as: every $H$ 
as in the theorem is a filtered deformation of a 
polynomial algebra. Etingof asks in 
\cite[Question 1.1]{Et} whether \emph{every} 
filtered deformation of a commutative domain in 
positive characteristic has to satisfy a polynomial 
identity. Taken together, Theorems \ref{xxthm0.1} 
and \ref{xxthm0.2} give some support for a positive 
answer to this question.

\subsection{}
\label{xxsec0.3} 
In Section 5, we classify IHOEs in the lowest dimensions. 
As noted in Lemma \ref{xxlem5.1}, there is only one 1-step 
IHOE, whatever the characteristic of the field, namely 
$\Bbbk [x]$ with $x$ primitive. But in 
dimension two the contrast between characteristic 0 
and positive characteristic is stark. Zhuang showed 
in \cite[Proposition 7.6]{Zh} that in characteristic 0, 
there are only two connected Hopf algebras of 
Gel'fand-Kirillov dimension two, and both are 
cocommutative IHOEs: namely $\Bbbk [x,y]$ with $x$ 
and $y$ primitive, and the enveloping algebra of 
the 2-dimensional non-abelian Lie algebra. In 
positive characteristic, however, it's a different 
story, as we now explain. 

Let ${\bf d_s} = \{d_s\}_{s\geq 0}$, ${\bf b_s} = 
\{b_s\}_{s\geq 0}$ and ${\bf c_{s,t}} = 
\{c_{s,t}\}_{0\leq s<t}$ be sequences of scalars
in $\Bbbk$ with only finitely many nonzero elements. 
Let $H({\bf d_s,b_s,c_{s,t}})$ denote the Ore extension
$\Bbbk[X_1][X_2;\mathrm{Id},\delta]$ where 
\begin{equation}
\label{E0.2.1}\tag{E0.2.1}
\delta(X_1)=\sum_{s\geq 0} d_s X_1^{p^s}.
\end{equation}
Define maps $\Delta$, $\epsilon$, and $S$ on $\{X_1, X_2\}$
by setting
$$\begin{aligned}
\Delta(X_1)&=X_1\otimes 1+1\otimes X_1,\\
\Delta(X_2)&=X_2\otimes 1+1\otimes X_2+ w,\\
\epsilon(X_1)&=\epsilon(X_2)=0,\\
S(X_1)&=-X_1,\\
S(X_2)&=-X_2-m(\mathrm{Id}\otimes S)(w),
\end{aligned}
$$ 
where $m$ denotes the multiplication operator, and
\begin{align}
\label{E0.2.2}\tag{E0.2.2}
w&=\sum_{s\geq 0} b_s 
  \left(\sum_{i=1}^{p-1} \frac{(p-1)!}{i! (p-i)!} (X_1^{p^s})^{i}
  \otimes (X_1^{p^s})^{p-i}\right)\\
	\notag
&\qquad \qquad+\sum_{0 \leq s<t} c_{s,t} \left(X_1^{p^s}\otimes 
  X_1^{p^t}- X_1^{p^t}\otimes X_1^{p^s}\right).
\end{align}
  
\begin{theorem}
[Propositions \ref{xxpro5.6} and \ref{xxpro5.11}]
\label{xxthm0.3}  
Let $p$ and $\Bbbk$ be as in Theorem \ref{xxthm0.1}. 
Let $H$ be a 2-step IHOE over $\Bbbk$, so 
$H = \Bbbk[X_1][X_2;\sigma,\delta]$ for 
$\sigma \in \mathrm{Aut}_{\Bbbk-\mathrm{alg}}(\Bbbk [X_1])$ 
and a $\sigma$-derivation $\delta$. 
\begin{enumerate}
\item[(1)]{\rm (Proposition \ref{xxpro5.6}(1))} 
For all choices of the scalars 
${\bf d_s,b_s, c_{s,t}}$, $H({\bf d_s,b_s,c_{s,t}})$ is 
a Hopf algebra.
\item[(2)] {\rm (Proposition \ref{xxpro5.6}(2))}
$H$ is isomorphic to $H({\bf d_s,b_s,c_{s,t}})$ for 
some choice of scalars
$\{d_s\}_{s\geq 0}$, $\{b_s\}_{s\geq 0}$ and 
$\{c_{s,t}\}_{0\leq s<t}$.
\item[(3)] {\rm (Proposition \ref{xxpro5.11})}
Two such $H({\bf d_s,b_s,c_{s,t}})$ and 
$H({\bf d'_s,b'_s,c'_{s,t}})$ are isomorphic as Hopf 
algebras if and only if there are nonzero scalars 
$\alpha,\beta$ in $\Bbbk$ such that,
for all the sequences of scalars,
\begin{equation}
\label{E0.3.1}\tag{E0.3.1}
d'_s=d_s \alpha^{p^s-1}\beta^{-1}, \quad 
b'_s=b_s \alpha^{p^{s+1}}\beta^{-1}, \quad
c'_{s,t}=c_{s,t} \alpha^{p^s+p^t}\beta^{-1}.
\end{equation}
If \eqref{E0.3.1} holds, there is a Hopf algebra 
isomorphism 
$\phi:H({\bf d_s,b_s,c_{s,t}})\to 
H({\bf d'_s},{\bf b'_s},{\bf c'_{s,t}})$ such that
$$\phi(X_1)=\alpha X_1, \quad {\text{and}}\quad \phi(X_2)=\beta X_2.$$
\end{enumerate} 
\end{theorem}

Restricting the above result to the cases where $H$ 
is commutative, that is where ${\bf d_s} = {\bf 0}$, 
yields the classification of connected unipotent 
algebraic groups over $\Bbbk$ of dimension 2. This is discussed in 
$\S$\ref{xxsec6.4}.

\subsection{}
\label{xxsec0.4} 
We study further properties of the 2-step IHOEs over 
$\Bbbk$ in $\S$\ref{xxsec8}. Thus we describe their 
antipodes and determine the Calabi-Yau members of 
the family in Proposition \ref{xxpro8.1}, and 
examine the finite dimensional representation theory 
of all the 2-step $\Bbbk$-IHOEs in Proposition 
\ref{xxpro8.2}, specifying the Azumaya locus, 
and determining all the simple modules and the 
extensions between them. 

The final subsection, $\S$\ref{xxsec8.3}, is 
inspired by another seminal paper of Jacobson, 
\cite{Ja2}, where he showed that every restricted 
Lie algebra $\mathfrak{g}$ of finite dimension $n$ 
has a \emph{restricted enveloping algebra} 
$u(\mathfrak{g})$, a Hopf algebra of dimension 
$p^n$ which is a Hopf factor of $U(\mathfrak{g})$. 
We show that a similar phenomenon occurs for 
the 2-step IHOEs. The following is an abbreviated 
version of Theorem \ref{xxthm8.8} and Propositions 
\ref{xxpro8.9}-\ref{xxpro8.10} together.

\begin{theorem}
\label{xxthm0.4} 
Let $H = H({\bf d_s}, {\bf b_s}, {\bf c_{s,t}})$ be 
a 2-step IHOE over $\Bbbk$ and assume that $H$ is not
commutative.
\begin{enumerate}
\item[(1)] 
$H$ contains a unique maximal central Hopf subalgebra 
$C(H)$, and $H$ is a free $C(H)$-module of rank $p^2$ when 
$b_0= 0$, and rank $p^3$ otherwise.
\item[(2)] 
Let $C(H)_{+}$ denote the augmentation ideal of $C(H)$. The 
resulting Hopf algebras $H/C(H)_{+} H$, of dimensions 
$p^2$ or $p^3$, are classified. 
\end{enumerate}
\end{theorem}

Finite dimensional connected Hopf $\Bbbk$-algebras have 
dimension a power of $p$, \cite[Proposition 1.1(1)]{Ma}. 
Those of dimension at most $p^3$ have been classified, 
in a series of papers \cite{Wa1, Wa2, NWW1, NWW2} by 
Xingting Wang and colleagues. We explain in 
$\S$\ref{xxsec8.3} how the Hopf algebras of 
Theorem \ref{xxthm0.4} fit into their classification.

\subsection{}\label{xxsec0.5} 
There have been a number of recent papers on infinite 
dimensional connected Hopf algebras of finite 
Gel'fand-Kirillov dimension, starting with \cite{Zh} 
and including the aforementioned \cite{BOZZ, Hua}, but 
also, for example, \cite{BGZ, BZ2, WZZ3, ZSL}. Much of 
this work has focused on the characteristic 0 setting 
(although \emph{not}, notably, \cite{Hua}, on which we 
rely heavily). Thus the present paper can be viewed 
as taking some further steps to open the topic in 
positive characteristic. As befits a gathering of 
first steps, the paper contains a number of open 
questions, scattered throughout. We can also point out 
here an open question of a more general, less precise 
type. Namely, it is possible that many or even all of 
the results proved here for positive characteristic 
IHOEs are valid, when rephrased appropriately, in the 
wider context of connected Hopf domains over $\Bbbk$ 
of finite Gel'fand-Kirillov dimension.

\subsection{}
\label{xxsec0.6} 
The paper is organized as follows. Section 1 contains 
some definitions and preliminaries, leading to Proposition 
\ref{xxpro1.2}, where it is shown that the automorphisms 
used to construct a positive characteristic IHOE have finite order, 
a key ingredient in the proof of Theorem \ref{xxthm0.1}. 
In Section 2, we use the Frobenius map and Noether's theorem 
on the finite generation of invariants to obtain the 
characterisation (Theorem \ref{xxthm2.3}) of when 
an iterated Ore extension in positive characteristic satisfies 
a PI. From this, it is easy to invoke Proposition \ref{xxpro1.2} 
to deduce Theorem \ref{xxthm0.1}(1) (Corollary \ref{xxcor2.4}). 
Theorem \ref{xxthm0.2} is proved in Section 3, using results of 
Chuang and Lee \cite{CL1}, depending ultimately on work of 
Kharchenko \cite{Kh} to obtain the needed information on the 
PI-degree of an iterated Ore extension which is known to satisfy 
a PI. In Section \ref{xxsec4}, we show that every IHOE has 
a filtration such that the associated graded ring is isomorphic
to the commutative polynomial algebra. The classification of 
2-step IHOEs is given in Section \ref{xxsec5}, including an 
analysis of their automorphism groups. A list of comments
and questions are collected in Section \ref{xxsec6}. A description 
of the Hopf center of all 2-step IHOEs is obtained in Section 
\ref{xxsec8}, enabling the construction and classification of 
their restricted Hopf algebra factors. The description of the 
Hopf center requires a special case of the noncommutative 
binomial theorem of Jacobson, which we recall in Section 
\ref{xxsec7}.

\subsection*{Acknowledgments}

The authors would like to thank Sara Billey, Jason Gaddis,
Sean Griffin, Liyu Liu and Xingting Wang for many useful 
conversations and valuable comments on the subject. The initial 
ideas for this paper were developed in March 2019 during the 
Oberwolfach miniworkshop ``Cohomology of Hopf Algebras and 
Tensor Categories''. We are extremely grateful to the 
Mathematisches Forschungsinstitut, and to the organisers of 
the workshop, for their hospitality.
%The authors thank
%the referee for his/her very careful reading and valuable comments.
K.A. Brown was supported by Leverhulme Emeritus Fellowship EM-2017-081-9
and J.J. Zhang by the US National Science Foundation (Nos. DMS-1700825 
and DMS-2001015).

%%%%%%%%%%%%%%%%%%%%%%%%%%%%%%%%%%%%%%%%%%%%%%%%%%%%%%%%%%%%%%%%%%%
\medskip

\section{Preliminaries}
\label{xxsec1}
\subsection{Definitions and their consequences}
\label{xxsec1.1} 
Throughout the paper  we shall use the standard notation $\Delta$, 
$\mu$, $\epsilon$, and $S$ for the comultiplication, 
multiplication, counit and antipode of a Hopf algebra $H$, with 
$\Delta(h) = h_1 \otimes h_2$ for $h \in H$. Unexplained Hopf 
algebra terminology can be found in \cite{Mo1}, for example. All 
Hopf algebras appearing in this paper will be (factors of) noetherian domains, 
so $S$ is necessarily bijective by \cite[Theorem A(ii)]{Sk}. For 
an algebra $A$, $Z(A)$ will denote the center of $A$, and if a 
group $G$ acts on $A$ as automorphisms, $A^G$ will denote the 
fixed subring of this action; when $G = \langle \sigma \rangle$ 
generated by single element $\sigma$, $A^{\langle \sigma\rangle}$ 
will be abbreviated to $A^{\sigma}$.

The following definition of the algebras in the paper's title 
is - on the face of it - significantly weaker than the original 
definition given in \cite[Definition 1.0]{Pa} and the modified 
definition in \cite[Definition 2.1]{BOZZ}. The improvement here 
is due to recent work of Huang \cite{Hua}, as we explain after 
the definitions. Recall that, given an $F$-algebra automorphism 
$\sigma$ of an $F$-algebra $R$, a $\sigma$-derivation $\delta$ 
of $R$ is an $F$-linear endomorphism of $R$ such that 
$\delta(ab) = \sigma(a)\delta(b) + \delta(a)b$ for all $a,b \in R$.

\begin{definition}
\label{xxdef1.1} Let $F$ be a field of any characteristic.
\begin{enumerate}
\item[(1)] 
Let $R$ be a Hopf $F$-algebra. A {\it Hopf Ore extension} 
(abbreviated to {\it HOE}) of $R$ is an algebra $H$ such that
\begin{enumerate}
\item[(1a)]
$H$ is a Hopf $F$-algebra with Hopf subalgebra $R$;
\item[(1b)]
there exist an algebra automorphism $\sigma$ of $R$ and a
$\sigma$-derivation $\delta$ of $R$ such that $H = R[x;\sigma,
\delta]$.
\end{enumerate}
\item[(2)] 
An {\it {\rm{(}}$n$-step{\rm{)}} iterated Hopf Ore extension of 
$F$} {\rm{(}}abbreviated to {\it {\rm{(}}$n$-step{\rm{)}} 
IHOE {\rm{(}}of $F${\rm{)}}}{\rm{)}} is a Hopf algebra 
\begin{equation} 
\label{E1.1.1}\tag{E1.1.1} 
H = F[X_1][X_2; \sigma_2 , \delta_2] \dots 
[X_n; \sigma_n , \delta_n],
\end{equation} 
where
\begin{enumerate}
\item[(2a)]
$H$ is a Hopf algebra;
\item[(2b)]
$H_{(i)} := F\langle X_1, \ldots , X_i \rangle$ is 
a Hopf subalgebra of $H$ for $i = 1, \ldots , n;$
\item[(2c)]
$\sigma_i$ is an algebra automorphism of $H_{(i-1)}$, 
and $\delta_i$ is a $\sigma_i$-derivation of $H_{(i-1)}$, 
for $i = 2, \ldots , n$.
\end{enumerate}
\end{enumerate}
\end{definition}

Throughout the paper, whenever $H$ is an HOE [respectively, 
an $n$-step IHOE], we assume that $H$ satisfies the 
conditions of Definition \ref{xxdef1.1}(1) 
[respectively, (2)], with the same notation.

The definition of an HOE in \cite[Definition 2.1]{BOZZ} 
required, in  addition to Definition \ref{xxdef1.1}(1a,1b), that
\begin{enumerate}
\item[(1c)]
there are $a,b\in R$ and $v, w\in R\otimes R$ such that
\begin{equation}
\label{E1.1.2}\tag{E1.1.2}
\Delta(x)=a\otimes x+x\otimes b+ v(x\otimes x)+ w.
\end{equation}
\end{enumerate}
However, by Huang's theorem  \cite[Theorem 1.3]{Hua},
if $R \subseteq T$ are noetherian $F$-algebras satisfying 
hypotheses (1a) and (1b) of Definition \ref{xxdef1.1}(1), and 
$R\otimes R$ is a domain, then, up to a change of variable,
\begin{enumerate}
\item[(1d)]
there are $a\in R$ and $w\in R\otimes R$ such that
\begin{equation}
\label{E1.1.3}\tag{E1.1.3}
\Delta(x)=a\otimes x+x\otimes 1+w.
\end{equation}
\end{enumerate}
This means that under the conditions in 
\cite[Theorem 1.3]{Hua}, namely that $R \otimes R$ is a 
noetherian domain, Definition \ref{xxdef1.1}(1) is equivalent 
to \cite[Definition 2.1]{BOZZ}. In turn, this yields the 
following consequences for IHOEs. Recall that a Hopf 
$F$-algebra is \emph{connected} if its coradical is 
$F$ - see e.g. \cite[Definition 5.1.5]{Mo1}.

\begin{proposition}
\label{xxpro1.2} 
Let $H$ be an IHOE of $F$. 
\begin{enumerate}
\item[(1)]
$H$ is noetherian and $H\otimes H$ is a domain.
\item[(2)]
After a change of variables {\rm{(}}but not of the 
subalgebras $H_{(i)}${\rm{)}}, 
\begin{equation}
\label{E1.2.1}\tag{E1.2.1}
\Delta(X_i)=1 \otimes X_i+ X_i\otimes 1+ w_i
\end{equation}
for some $w_i\in H_{(i-1)}\otimes H_{(i-1)}$, for 
$i = 1, \ldots , n$.
\item[(3)] 
$H$ is a connected Hopf algebra.
\item[(4)] 
For $i = 1, \ldots , n$ there is a character 
$\chi_i$ of $H_{(i-1)}$ such that
$$ \sigma_i = \tau^{\ell}_{\chi_i} = \tau^r_{\chi_i} $$
is a left and right winding automorphism of $H_{(i-1)}$. 
In particular for each $j$ with $j < i$ there exists 
$a_{ji} \in H_{(j-1)}$ such that
$$ \sigma_i(X_j) = X_j + a_{ji}. $$
\end{enumerate}
\end{proposition}

\begin{proof} (1) This follows by induction on $n$ using 
\cite[Theorem 1.2.9]{McR}.

\noindent 
(2) Fix $i$, $1 \leq i \leq n$ (where we take 
$H_{(0)} = F$, $\sigma_1 = \mathrm{id}_{F}$ and 
$\delta_1 = 0$). By (1), $H_{(i-1)}$ is noetherian and 
$H_{(i-1)}\otimes H_{(i-1)}$ is a domain. Hence, by 
\cite[Theorem 1.3]{Hua},
$$ \Delta(X_i) = a_i \otimes X_i + X_i \otimes 1 + w_i $$
for a group-like element of $a_i$ of $H_{(i-1)}$ and 
$w_i \in H_{(i-1)} \otimes H_{(i-1)}$. But $H$ has no
nontrivial invertible elements. Therefore, noting that 
$\mu \circ (\epsilon \otimes \mathrm{id})\circ \Delta 
= \mathrm{id}_{H_{(i-1)}},$ $a_i=1$, yielding \eqref{E1.2.1}.

\noindent (3) This is proved in \cite[Proposition 2.5]{BOZZ}.

\noindent (4) That the map $\sigma_i$ is a (left and right) 
winding automorphism (with the same character) of 
$H_{(i-1)}$ is proved in \cite[Theorem 1.2]{BOZZ} and its 
improvement \cite[Theorem 1.3]{Hua}. From \eqref{E1.1.3} 
applied to $X_j$ we see that
\begin{equation}
\label{E1.2.2}\tag{E1.2.2}
a_{ji} = \chi_i(X_j) + 
\sum {w_{j}}_{(1)}\chi_i({w_{j}}_{(2)}) \in H_{(j-1)}.
\end{equation}
\end{proof}

Since it is needed for the proof of Theorem 
\ref{xxthm0.1}, we recall a result in classical invariant 
theory due to E. Noether \cite{No}, also see 
\cite[Theorem 1.1]{Mo2} and \cite[Theorem 2.3.1]{Sm}.

\begin{theorem} 
\label{xxthm1.3}
Let $A$ be an affine commutative algebra over $F$ and
$G$ be a finite subgroup of $\mathrm{Aut}_{F-\mathrm{alg}}(A)$. 
Then the fixed subring $A^G$ is affine over $F$ and $A$ is a 
finitely generated module over $A^G$.
\end{theorem}

\subsection{Positive characteristic aspects}
\label{xxsec1.2} 
Given Proposition \ref{xxpro1.2}, the following 
result is  key to our main results.

\begin{proposition}
\label{xxpro1.4}
Suppose $\Bbbk$ has positive characteristic $p$.
Let $d$ be a positive integer.
\begin{enumerate}
\item[(1)]
Let $R$ be a Hopf $\Bbbk$-algebra such that the order of 
every left or right winding automorphism of $R$ divides 
$d$. Suppose $H:=R[x,\sigma,\delta]$ is an HOE of $R$ such 
that
\begin{equation}
\label{E1.4.1}\tag{E1.4.1}
\Delta(x)=1\otimes x+x\otimes 1+w
\end{equation}
for some $w\in R\otimes R$. Then the order of every 
left or right winding automorphism of $H$ divides $dp$.
\item[(2)]
Let $H$ be an $n$-step IHOE of $\Bbbk$. Then every left 
or right winding automorphism of $H$ has order dividing 
$p^n$.
\end{enumerate}
\end{proposition}

\begin{proof} 
We prove the results for left winding automorphisms.

\noindent 
(1) Let $\pi: H\to \Bbbk$ be an algebra map and let $\Xi^l_{\pi}$ 
be the corresponding left winding automorphism  of $H$, so 
$\Xi^l_{\pi}(h) = \pi(h_1)h_2$ for $h \in H$. Since $R$ 
is a Hopf subalgebra of $H$ and $\pi\mid_{R}$ is a character 
of $R$, $\Xi^l_{\pi}$ restricted to $R$, still denoted by 
$\Xi^l_{\pi}$, is a left winding automorphism of $R$. By 
assumption $(\Xi^l_{\pi})^{d}$ is the identity of 
$R$. It remains to show that $(\Xi^l_{\pi})^{dp}$ is the 
identity when applied to $x$. Using \eqref{E1.4.1},
$$\Xi^l_{\pi}(x)=x +\pi(x)+\mu \circ (\pi\otimes 1)(w)=x+s$$
where $s=\pi(x)+\mu \circ (\pi\otimes 1)(w)\in R$. Then 
$$(\Xi^l_{\pi})^{d}(x)=x+t$$
for some $t\in R$. Note that $(\Xi^l_{\pi})^{d}(t)=t$ as 
the order of $\Xi^l_{\pi}$ restricted to $R$ divides $d$. Then
$$(\Xi^l_{\pi})^{dp}(x)=((\Xi^l_{\pi})^{d})^{p}(x)=x+pt=x.$$
Therefore $(\Xi^l_{\pi})^{dp}$ is the identity map on $H$ 
as required.

\noindent 
(2) This follows from part (1), Proposition \ref{xxpro1.2}(2) 
and induction on $n$. 
\end{proof}

Note that there is an anti-monomorphism of groups from the 
character group $X(H)$ of $H$ (that is, the algebra 
homomorphisms from $H$ to $\Bbbk$) to the group of left 
winding automorphisms of $H$, \cite[Section 2.5]{BZ1}. Hence, 
viewing $H$ in dual language, as a quantum group, Proposition 
\ref{xxpro1.4}(2) can be interpreted as giving a bound on the 
exponent of the maximal classical subgroup of $H$, which is a 
unipotent algebraic group.

\medskip

\section{Proof of Theorem \ref{xxthm0.1}}
\label{xxsec2}

\subsection{HOEs and IHOEs in positive characteristic}
\label{xxsec2.1} 
If $C$ is a commutative $\Bbbk$-algebra, where $\Bbbk$ has 
positive characteristic $p$, we denote the Frobenius map 
$c \mapsto c^p$ on $C$ by $F$. In this section, we always 
assume the following. 

\begin{convention}
\label{xxcon2.1} 
When $\Bbbk$ is a field of positive characteristic and $C$
is a commutative $\Bbbk$-algebra, we use $F(C)$ to denote 
the $\Bbbk$-subalgebra of $C$ generated by the image of $C$. 
Thus, in general, $F(C)$ is larger than the image of $C$ under $F$.
\end{convention}

Given an automorphism $\sigma$ and a $\sigma$-derivation 
$\delta$ of an algebra $A$, a subspace $E$ of $A$ is called 
\emph{$(\sigma,\delta)$-trivial} if $\sigma(e) = e$ and 
$\delta(e) = 0$ for all $e \in E$.

\begin{lemma}
\label{xxlem2.2}
Let $B$ be an Ore extension $A[x;\sigma,\delta]$ of an 
affine $\Bbbk$-algebra $A$.
\begin{enumerate}
\item[(1)]
Suppose that no nonzero element of $Z(A)$ is a zero 
divisor on $A$, and that $Z(A)^{\sigma} \subsetneq Z(A)$. 
Then 
$$z x=xz$$
for all $z\in Z(A)^{\sigma}$. As a consequence, $B$ contains 
the commutative subalgebra $Z(A)^{\sigma}[x]$.
\item[(2)]
Suppose $Z(A)=Z(A)^{\sigma}$ and ${\text{char}}\; \Bbbk=p>0$.
Then 
$$ z^p x= x z^p$$
for all $z\in Z(A)$. As a consequence, $B$ contains the 
commutative subalgebra $F(Z(A))[x]$.
\item[(3)]
Suppose that
\begin{enumerate}
\item[(a)] 
$A$ is a prime ring that is a finitely generated module 
over $Z(A)$;
\item[(b)]
$Z(A)$ is an affine $\Bbbk$-algebra;
\item[(c)]
$\sigma\mid_{Z(A)}$ has finite order;
\item[(d)]
${\text{char}}\; \Bbbk=p>0$.
\end{enumerate}
Then $B$ is a prime noetherian algebra satisfying a 
polynomial identity, and $A$ is a finite module over 
a $(\sigma, \delta)$-trivial subalgebra of its center.
\item[(4)]
Suppose, in addition to hypotheses (3)(a)-(3)(d), that
\begin{enumerate}
\item[(e)] $A$ is a maximal order.
\end{enumerate}
Then $B$ is a maximal order and is a finitely generated module 
over $Z(B)$, which is an affine normal $\Bbbk$-algebra.
\end{enumerate}
\end{lemma}

\begin{proof} (1) Let $z \in Z(A)^{\sigma}$. Then 
$xz = zx + \delta(z)$, so we only need to show that $\delta(z)=0$ 
for all $z\in Z(A)^{\sigma}$. Pick any element $y\in Z(A)
\setminus Z(A)^{\sigma}$. Then $zy=yz$. Applying $\delta$ to 
this equation we have
$$\delta(z) y+\sigma(z) \delta (y)
=\delta(y) z +\sigma(y)\delta(z).$$
Since $z\in Z(A)^\sigma$, we have 
$\sigma(z) \delta (y)= z \delta(y)=\delta(y) z,$ 
which implies that
$$\delta(z) y=\sigma(y)\delta(z).$$
Hence, since $\sigma(y) \in Z(A)$,
$$\delta(z) (y-\sigma(y))=0.$$
But $y-\sigma(y) \in Z(A) \setminus \{0\}$ and is thus, 
by hypothesis, not a zero divisor in $A$. So 
$\delta(z)=0$ as required,  and hence $Z(A)^{\sigma}[x]$ 
is commutative.

\noindent 
(2) Suppose that $Z(A)=Z(A)^{\sigma}$. Then, for every 
$z\in Z(A)$,
$$x z= \sigma(z) x+\delta(z)=z x+ \delta(z).$$
By induction on $n$ and since $z \in Z(A)$, we have, 
for all $n \geq 1$,
$$xz^n =z^n x+n z^{n-1} \delta(z).$$
The assertion follows because $pz^{p-1}\delta(z)=0$. Once 
again, the consequence is clear.

\noindent 
(3) It is clear from \cite[Theorem 1.2.9]{McR} and 
hypotheses {\it (a)} and {\it (b)} that $B$ is noetherian, and it is 
prime by \cite[Theorem 1.2.9(iii)]{McR}. By Theorem 
\ref{xxthm1.3}, and hypotheses {\it (a)} and {\it (c)}, $A$ is a finitely 
generated $Z(A)^{\sigma}$-module. When 
$Z(A)\neq Z(A)^{\sigma}$,  $B$ is a finitely generated 
left module over its commutative affine subalgebra 
$Z(A)^{\sigma}[x]$ by part (1). When $Z(A)=Z(A)^{\sigma}$, 
$B$ is a finitely generated left module over its 
commutative affine subalgebra $F(Z(A))[x]$ by part (2). 
In both cases, $B$ thus satisfies a polynomial identity, 
by \cite[Corollary 13.4.9(i)]{McR}. The second claim in 
the final sentence of (3) is clear from (1), (2) and 
Theorem \ref{xxthm1.3}.

\noindent 
(4) Suppose that $A$ satisfies hypothesis {\it (e)}. Then $B$ 
is also a maximal order by \cite[Proposition V.2.5]{MaR}. 
Therefore, since it is a prime noetherian PI ring by 
part (3), $B$ equals its own trace ring and is thus  
a finite module over its affine normal center by 
\cite[Propositions 13.9.8(i) and 13.9.11(ii)]{McR}.
\end{proof}

Theorem \ref{xxthm0.1} is a corollary of the following 
more general result applying to iterated Ore extensions 
in positive characteristic.

\begin{theorem}
\label{xxthm2.3}
Suppose $\Bbbk$ has positive characteristic $p$, let 
$n$ be a non-negative integer, and let $R$ be an 
$n$-step iterated Ore extension,
\begin{equation} 
\label{E2.3.1}\tag{E2.3.1} 
R = \Bbbk[X_1][X_2; \sigma_2 , \delta_2] \dots 
[X_n; \sigma_n , \delta_n].
\end{equation}
For $i = 1, \ldots , n-1$, denote 
$\Bbbk \langle X_1, \ldots , X_i \rangle$ by $R_{(i)}$.
\begin{enumerate}
\item[(1)] 
$R$ satisfies a polynomial identity if and only if 
${\sigma_{i}}\mid_{Z(R_{(i-1)})}$ has finite order 
for all $i = 2, \ldots , n$.
\item[(2)] 
Suppose that $R$ satisfies a polynomial identity. 
Then it is a finitely generated module over its 
center, which is a normal $\Bbbk$-affine domain. 
Moreover $R$ is a homologically homogeneous and 
GK-Cohen Macaulay domain, with
$$\gldim R = \GKdim  R = n. $$
\end{enumerate} 
\end{theorem}

\begin{proof} 
(1) Suppose that $R$ satisfies a PI. Then so also 
does $R_{(i)} = R_{(i-1)}[X_i ; \sigma_i, \delta_i]$ 
for $i =2, \ldots , n$. For each such $i$, since 
$R_{(i-1)}$ is a domain satisfying a PI, 
\cite[Theorem 2.7]{LeM} then requires that 
$\left| {\sigma_{i}}\mid_{ Z(R_{(i-1)})}\right| < \infty$, 
as claimed.

Conversely, assume that each ${\sigma_{i}}\mid_{Z(R_{(i-1)})}$ 
has finite order. By induction on $n$ we may assume that 
$A := R_{(n-1)}$ satisfies a PI, and that part (2) of the 
theorem has been proved for $A$. In particular, $A$ is a 
finitely generated $Z(A)$-module. Since $A$ is an iterated 
Ore extension of a field, it is a domain. By our induction 
hypothesis and since ${\sigma_{n}}\mid_{Z(A)}$ has finite 
order, hypotheses (3){\it (a)-(d)} of Lemma \ref{xxlem2.2} are 
satisfied. By \cite[Proposition V.2.5]{MaR}, $A$ is a 
maximal order, so {\it (e)} of Lemma \ref{xxlem2.2}(4) holds. 
Therefore, by Lemma \ref{xxlem2.2}(4), $R$ is a finitely 
generated $Z(R)$-module and $Z(R)$ is an affine normal domain. 

\noindent
(2) The statements in the second sentence have already been 
proved in (1). The properties listed in the last sentence 
are also proved by induction on $n$. By the induction 
hypothesis, $\gldim A = n-1$, so $\gldim R$ is $n -1$ or $n$ 
by \cite[Theorem 7.5.3(i)]{McR}. Thanks to 
\cite[Corollary 7.1.14]{McR}, there is a simple (left) 
$A$-module $V$ with $\projdim V = n-1$. Choose a maximal 
left ideal $I$ of $A$ with $V \cong A/I$. Let $W$ be a simple 
$R$-module which is a factor of the nonzero cyclic $R$-module 
$R/RI \cong R \otimes_A V$, so $W \cong R/L$ for a left ideal 
$L$ of $R$ containing $RI$. Since $R$ is a noetherian affine 
$\Bbbk$-algebra satisfying a polynomial identity, by part (1),  
$\dim_{\Bbbk} W < \infty$ by Kaplansky's theorem, 
\cite[Theorem 13.10.3(i)]{McR}. In particular, $W$ is a 
finitely generated $A$-module, and $L \cap A = I$ by maximality 
of $I$. Thus $V$ is an $A$-submodule of $W$. Hence, choosing 
an $A$-module $X$ such that $\mathrm{Ext}^{n-1}_A(V,X) \neq 0$, 
the long exact sequence of $\Ext$ yields
$$\Ext^{n-1}_A(W,X)\longrightarrow\Ext^{n-1}_A(V,X)\longrightarrow 0, $$
so that $\projdim_A W = n-1$. Therefore $\gldim R = n$ by 
\cite[Corollary 7.9.18]{McR}. 

Since $A$ is homologically homogeneous, so is $R$, by 
\cite[Theorem 2.3]{Zho}, noting that the hypotheses of Zhong's 
theorem are satisfied in view of Lemma \ref{xxlem2.2}(1),(2), 
which guarantee that $A$ is a finite module over a central 
subalgebra which is $(\sigma_n, \delta_n)$-trivial. Finally 
$R$, being a homologically homogeneous algebra which is a 
finite module over an affine central domain, is 
GK-Cohen Macaulay by \cite[Corollary 5.4 and Theorem 4.8]{BM}.
\end{proof}

Theorem \ref{xxthm0.1} is now an immediate consequence of 
the above theorem combined with Proposition \ref{xxpro1.4}(2):

\begin{corollary}
\label{xxcor2.4}
Suppose $\Bbbk$ has positive characteristic $p$ and let 
$n$ be a non-negative integer. Then every $n$-step IHOE 
$H$ of $\Bbbk$ is a finitely generated module over its 
center $Z(H)$ and $Z(H)$ is a normal $\Bbbk$-affine domain. 
Moreover $H$ is a homologically homogeneous and GK-Cohen 
Macaulay algebra with
$$\gldim H = \GKdim  H = n. $$
\end{corollary}

\subsection{First questions around Theorem \ref{xxthm0.1}}
\label{xxsec2.2} 
We briefly consider two possible improvements of Theorem 
\ref{xxthm0.1}.

\begin{remarks}
\label{xxrem2.5} 
(1) As a large family of examples to which Corollary 
\ref{xxcor2.4} applies, let $\mathfrak{g}$ be a finite 
dimensional completely solvable Lie algebra over a field 
$\Bbbk$ of positive characteristic $p$. (That is, 
$\mathfrak{g}$ has a full flag of ideals.) Thus the 
enveloping algebra $U(\mathfrak{g})$ is an IHOE, so 
Corollary \ref{xxcor2.4} tells us that $U(\mathfrak{g})$ 
is a finite module over its center. This is a special case 
of a fundamental result due to Jacobson \cite{Ja1}:

\begin{enumerate}
\item[]
\emph{Let $\mathfrak{g}$ be a finite dimensional Lie algebra 
over $\Bbbk$. Then $U(\mathfrak{g})$ contains a central Hopf 
subalgebra over which $U(\mathfrak{g})$ is a free module of 
finite rank equal to a power of $p$.}
\end{enumerate}

\noindent
Note that Jacobson's theorem does not require $\mathfrak{g}$ 
to be restricted. In language not then available to Jacobson, 
his result implies that the PI-degree of $U(\mathfrak{g})$ is 
a power of $p$, prefiguring Theorem \ref{xxthm3.3} below. 
Moreover, the Lie algebra case thus suggests a possible 
strengthened version of Corollary \ref{xxcor2.4}: 

\begin{question}
\label{xxque2.6} 
Is every IHOE over a field of positive characteristic a 
finite module over a central Hopf subalgebra?
\end{question}

(2) Since the enveloping algebra of a finite dimensional Lie 
algebra $\mathfrak{g}$ is not an IHOE when $\mathfrak{g}$ 
contains a non-abelian (classical) simple subalgebra\footnote{
See \cite[Example 3.1(iv)]{BOZZ}, where the 
characteristic 0 hypothesis is not necessary.} 
other than $\mathfrak{sl}(2, \Bbbk )$, one might wonder, in 
the light of Jacobson's theorem discussed in (1), whether 
\emph{every} connected affine Hopf algebra of finite 
Gel'fand-Kirillov dimension is finite over 
its center, or even finite over a central Hopf subalgebra, 
when $\Bbbk$ has positive characteristic.  As a first test 
case for this, consider an affine connected graded Hopf 
algebra $H$, meaning that 
\begin{equation}
\label{E2.6.1}\tag{E2.6.1}
H = \bigoplus_{i \geq 0} H_i
\end{equation}
is connected graded as both an algebra \emph{and} as a 
coalgebra. Two recent papers, \cite{BGZ} and \cite{ZSL}, 
study such Hopf algebras. (In \cite{BGZ} the possibly 
weaker hypothesis that $H$ is only connected graded 
\emph{as an algebra} is also treated, but we do not 
discuss that here.) Theorem A of \cite{ZSL}
states that if $\Bbbk$ has characteristic 0 
and $H$ is an affine connected graded Hopf $\Bbbk$-algebra, 
then $H$ is an $n$-step IHOE, where $n$ is the 
Gel'fand-Kirillov dimension of $H$. The hypothesis on 
$\Bbbk$ is necessary here - consider 
$H = \Bbbk [ x]/ \langle x^p \rangle$ where $\Bbbk$ has 
characteristic $p$. Nevertheless, Corollary \ref{xxcor2.4} 
and \cite[Theorem A]{ZSL} prompt the 

\begin{question}
\label{xxque2.7}
Suppose $H$ is an affine connected graded Hopf $\Bbbk$-algebra 
of finite Gel'fand-Kirillov dimension, where $\Bbbk$ has 
positive characteristic. Is $Z(H)$ affine, and is $H$ a 
finite $Z(H)$-module? 
\end{question}
\end{remarks}

\medskip

\section{The PI degree and other invariants}
\label{xxsec3}

\subsection{PI-degree}
\label{xxsec3.1} 
Theorem \ref{xxthm3.3}, the main result of this subsection, gives 
information on the PI-degree of an $n$-step IHOE in characteristic 
$p$. Recall that the \emph{PI-degree} of a prime PI ring $R$ is 
denoted by $\PIdeg R$ and defined to be $d$, where $d^2$ 
is the dimension over the field $Z(Q(R))$ of the central 
simple quotient ring $Q(R)$, see \cite[13.3.6, 13.6.7]{McR} 
and \cite[p.115]{BG}. When $R$ is in addition an affine 
$\Bbbk$-algebra with the field $\Bbbk$ algebraically closed, 
$\PIdeg R$ equals the maximum dimension over $\Bbbk$ of the 
simple $R$-modules \cite[Theorem 13.10.3]{McR}, 
\cite[Theorem 1.13.5(2)]{BG}. To prove Theorem 
\ref{xxthm3.3} we need the following definitions and 
theorem, adapted from \cite{CL1}. Note that the definitions 
in \cite{CL1} are expressed in terms of the left Martindale 
quotient ring of a ring $R$, but when $R$ is a prime PI ring, 
as here, this coincides with the total quotient ring 
$Q(R) = R[Z(R)\setminus \{0\}]^{-1}$, as follows easily 
from the definition of the Martindale quotient ring and 
Posner's theorem, \cite[Sect. 10.3.5 and Theorem 13.6.5]{McR}. 

\begin{definition} \cite{CL1}
\label{xxdef3.1} 
Let $R$ be a prime algebra satisfying a polynomial 
identity, $\sigma \in \mathrm{Aut}_{\Bbbk-\mathrm{alg}}(R)$ 
and $\delta$ a $\sigma$-derivation of $R$. Let $Q$ be the 
total quotient ring of $R$.
\begin{enumerate}
\item[(1)] 
 $\sigma$ is called {\it X-inner} if its extension to $Q$
is inner, that is, there exists a unit 
$b$ of $Q$ such that 
$$\sigma(x)=b^{-1}xb$$
for all $x\in Q$. Otherwise, $\sigma$ is called {\it X-outer}.
\item[(2)]
$\sigma$ is called {\it quasi-inner} if there exists an
integer $n\geq 1$ such that $\sigma^n$ is X-inner. The least 
such integer $n$ is called the {\it outer degree} of $\sigma$ 
and is denoted by ${\mathrm{Outdeg}}\; \sigma$.
\item[(3)] 
$\delta$ is called {\it X-inner} if its extension to $Q$
is inner, that is, there exists $b\in Q$ such that 
$$\delta(x)=bx-\sigma(x) b$$
for all $x\in Q$. Otherwise, $\delta$ is called {\it X-outer}.
\item[(4)] 
$\delta$ is called {\it quasi-algebraic} if there exist a 
positive integer $n$, an automorphism $g$ of $Q$ and 
$b_1, \cdots, b_{n-1}, b\in Q$ such that for all $x\in R$,
$$\delta^n (x) + b_1 \delta^{n-1}(x) +\cdots 
+ b_{n-1} \delta(x)=bx-g(x) b.$$
The least such integer $n$ is called the 
{\it quasi-algebraic degree} or the {\it outer degree} of 
the $\sigma$-derivation $\delta$ and is 
denoted by ${\mathrm{Outdeg}}\; \delta$. In particular, 
${\mathrm{Outdeg}}\; \delta= 1$ if and only if $\delta$ 
is X-inner.
\end{enumerate}
\end{definition}

The main result of Chuang-Lee \cite{CL1} is the following.

\begin{theorem}\cite[Theorem 2.5]{CL1}
\label{xxthm3.2} 
Let $R$ be a prime PI algebra, $\sigma \in 
\mathrm{Aut}_{\Bbbk - \mathrm{alg}}(R)$ and $\delta$ a 
$\sigma$-derivation of $R$. Then $R[x;\sigma, \delta]$ is a 
PI ring if and only if $\delta$ is quasi-algebraic and 
$\sigma$ is quasi-inner. In this case,
\begin{enumerate}
\item[(a)] 
if $\delta$ is X-outer, then 
$\PIdeg  R[x;\sigma, \delta] = \PIdeg  R \times 
\mathrm{Outdeg}\;\delta$;
\item[(b)] if $\delta$ is X-inner, then 
$\PIdeg  R[x; \sigma, \delta] =\PIdeg  R \times 
\mathrm{Outdeg}\;\sigma$.
\end{enumerate}
\end{theorem}

We can now deduce our second main theorem from 
Corollary \ref{xxcor2.4} and Theorem \ref{xxthm3.2}:

\begin{theorem}
\label{xxthm3.3} 
Suppose that $\Bbbk$ is a field of characteristic 
$p > 0$, and let $H$ be an IHOE over $\Bbbk$. Then 
$\PIdeg H$ is a power of $p$.
\end{theorem}

\begin{proof} 
We argue by induction on $n$, where $H$ is an 
$n$-step IHOE. When $n = 1$ $H = \Bbbk [x]$ and the 
result is clear. For the induction step, write $B$ 
for $H_{(n-1)}$, so, by the induction hypothesis,
\begin{equation}
\label{E3.3.1}\tag{E3.3.1} 
\PIdeg  B = p^t,
\end{equation}
for some $t \geq 0$. After relabeling, 
$H = B[x; \sigma, \delta]$ for 
$\sigma \in \mathrm{Aut}_{\Bbbk-\mathrm{alg}}(B)$ and 
a $\sigma$-derivation $\delta$ of $B$. By Proposition 
\ref{xxpro1.2}(1) and Corollary \ref{xxcor2.4}, $B$ and 
$H$ are both prime PI-algebras. Thus, by Theorem 
\ref{xxthm3.2}, $\sigma$ is quasi-inner and $\delta$ 
is quasi-algebraic, and there are two cases to 
consider.

\noindent 
{\bf Case $(a)$:} $\delta$ is X-outer. Then, since $B$ 
is a prime PI ring, a result of Kharchenko, see  
\cite[Lemmas 5 and 6]{CL2}, guarantees the existence of a unit 
$u$ of the quotient ring $Q(B)$ and an X-outer derivation 
$d$ of $Q(B)$ such that, for all $b \in B$,
$$ \sigma (b) = ubu^{-1} \qquad {\text{and}} \qquad 
\delta = ud. $$
Note that $\sigma$ and $\delta$ extend respectively to 
an automorphism and a $\sigma$-derivation of $Q(B)$. 
Then, by \cite[Lemma 1.3(2)]{CL1},
\begin{equation}
\label{E3.3.2} \tag{E3.3.2} 
Q(B)[x; \sigma, \delta] \cong Q(B)[x; d].
\end{equation}
Now
\begin{equation}
\label{E3.3.3}\tag{E3.3.3} 
\PIdeg B = \PIdeg Q(B)
\end{equation}
and
\begin{equation}
\label{E3.3.4}\tag{E3.3.4} 
\PIdeg H = \PIdeg Q(B)[x;\sigma, \delta].
\end{equation}
To prove the induction step in Case (a) it is therefore 
enough, by \eqref{E3.3.1}, \eqref{E3.3.2}, \eqref{E3.3.3} 
and \eqref{E3.3.4}, to prove that
\begin{equation}
\label{E3.3.5}\tag{E3.3.5}
\PIdeg Q(B)[x; d] = \PIdeg B \times p^{\ell}
\end{equation}
for some $\ell \geq 0$. But \eqref{E3.3.5} follows 
immediately from Theorem \ref{xxthm3.2}(a) and a 
theorem of Kharchenko, stated and proved as 
\cite[Theorem, p.60]{CL2}.

\noindent 
{\bf Case $(b)$:} $\delta$ is X-inner. Now Theorem 
\ref{xxthm3.2}(b) applies. But clearly, from the 
definition,
$$ \mathrm{Outdeg}\;  \sigma \mid |\sigma|. $$
Moreover, by \cite[Theorem 1.3(ii)]{Hua} and Proposition 
\ref{xxpro1.4}(2),
$$ |\sigma| \mid p^n.$$
Therefore (\ref{E3.3.1}) is proved in this case also, 
and so the theorem follows. 
\end{proof}

\begin{remarks}
\label{xxrem3.4} 
(1) A second proof of Theorem \ref{xxthm3.3} will be 
given Section \ref{xxsec4}, where we obtain it as a 
corollary of Theorem \ref{xxthm4.3} together with a 
result of Etingof \cite{Et} on filtered deformations 
of commutative domains in positive characteristic.

\noindent 
(2) Faced with Theorem \ref{xxthm3.3} it is 
natural to ask

\begin{question}
\label{xxque3.5}
What is the power of $p$ occurring in Theorem 
\ref{xxthm3.3}?
\end{question}

This is a delicate matter, as can be seen from 
the case of enveloping algebras. Assume that 
$\Bbbk$ is algebraically closed of characteristic 
$p > 0$, and let $\mathfrak{g}$ be a finite 
dimensional Lie $\Bbbk$-algebra. For 
$f\in \mathfrak{g}^*$ define
$$\mathrm{stab}_{\mathfrak{g}}(f) 
:= \{ x \in \mathfrak{g} \mid f([x,-]) = 0 \}, $$
and set
$$\mathrm{ind} \mathfrak{g} 
:= \mathrm{min} \{ \mathrm{dim}_{\Bbbk} 
\mathrm{stab}_{\mathfrak{g}} (f) \mid 
f \in \mathfrak{g}^* \}. 
$$

The first Kac-Weisfeiler conjecture \cite{KW} 
proposed that when $\mathfrak{g}$ is 
restricted,
\begin{equation}
\label{E3.5.1}\tag{E3.5.1} 
\mathrm{PIdeg}U(\mathfrak{g}) = 
p^{\frac{1}{2}(\mathrm{dim}\mathfrak{g} - 
\mathrm{ind}\mathfrak{g})}.
\end{equation}
This was already known, due to Rudakov \cite{Ru}, 
when $\mathfrak{g}$ is the Lie algebra of a 
reductive group (in this case $\mathrm{ind}\mathfrak{g}$ 
is the rank of $\mathfrak{g}$); and Strade \cite{St} 
proved the conjecture for $\mathfrak{g}$ solvable in 
1978, extending the completely solvable case done in 
\cite{KW}. Premet and Skryabin \cite{PS} confirmed the 
conjecture for all restricted $\mathfrak{g}$ admitting 
a toral stabilizer, and also showed that 
$p^{\frac{1}{2}(\mathrm{dim}\mathfrak{g} - 
\mathrm{ind}\mathfrak{g})}$ 
is a lower bound for the PI-degree for all restricted 
$\mathfrak{g}$. The general case, however, remains open, 
although it has recently been confirmed in \cite{MSTT} 
for all restricted subalgebras $\mathfrak{g}$ of 
$\mathfrak{gl}_n (\Bbbk)$, for $p \gg 0$. Note, 
however, that the conjecture is not true if the 
hypothesis that $\mathfrak{g}$ is restricted is 
omitted - Topley \cite{To} exhibits pairs of Lie 
$\Bbbk$-algebras $\mathfrak{g}_1$ and $\mathfrak{g}_2$, 
both solvable, with $\mathfrak{g}_1 \ncong \mathfrak{g}_2$ 
and indeed $\mathrm{ind}\mathfrak{g}_1 \neq 
\mathrm{ind}\mathfrak{g}_2$, but with 
$U(\mathfrak{g}_1) \cong U(\mathfrak{g}_2)$.

It might be easier to approach Question \ref{xxque3.5} 
by first seeking an upper bound:

\begin{question}
\label{xxque3.6}
Is the PI-degree of an $n$-step IHOE over a field 
$\Bbbk$ of characteristic $p$ bounded above by 
$p^{\frac{n}{2}}$? 
\end{question}

We show in $\S$\ref{xxsec8.2} that this is true when $n\leq 2$.

\noindent 
(3) As recalled at the start of $\S$\ref{xxsec3.1}, 
when $\Bbbk$ is algebraically closed the PI degree 
is the maximum dimension of a simple $H$-module, 
so Question \ref{xxque3.5} leads one to the broader 
issue of the possible dimensions of all the simple 
$H$-modules, when $H$ is an $n$-step IHOE over an 
algebraically closed field of characteristic $p$. 
One might hope that every simple module has 
dimension a power of $p$, and we confirm that this 
is the case for all $2$-step IHOEs in Proposition 
\ref{xxpro8.2}. Further positive evidence is 
provided by the enveloping algebras of completely 
solvable Lie algebras. Such a Lie algebra 
$\mathfrak{g}$ of dimension $n$  has a chain of 
$n+1$ ideals from $\{0\}$ to $\mathfrak{g}$, 
so that $U(\mathfrak{g})$ is an $n$-step IHOE. Kac 
and Weisfeler showed in \cite{KW} that the dimension 
of every simple $U(\mathfrak{g})$-module is a power 
of $p$. But the naive hope fails for 
$\mathfrak{sl}(2, \Bbbk)$ when $\Bbbk$ has 
characteristic $p>2$, since $U(\mathfrak{sl}(2,\Bbbk))$ 
is a 3-step IHOE \cite[Example 3.1(iv)]{BOZZ}, and 
Rudakov and Shafarevitch showed \cite{RS} that the 
dimensions of the simple 
$U(\mathfrak{sl}(2,\Bbbk))$-modules are $1,\ldots,p$. 
These examples and results suggest the following 
possibility:

\begin{question}
\label{xxque3.7}
Suppose $\Bbbk$ is an algebraically closed field of 
characteristic $p > 0$. Let $H$ be an $n$-step IHOE 
over $\Bbbk$, which has a chain of Hopf subalgebras 
as in \eqref{E1.1.1} with each $\sigma_i$ being the 
identity. Is the dimension of every simple $H$-module 
a power of $p$?
\end{question}

The above question has a positive answer for all 
2-step IHOEs - see Proposition \ref{xxpro8.2}.
\end{remarks}

\subsection{The antipode and Nakayama automorphism}
\label{xxsec3.2} 
The antipode $S$ of a Hopf algebra that is a finite 
module over its center has finite order up to inner 
automorphisms by \cite[Corollary 0.6(c)]{BZ1}. In the 
case of IHOEs in positive characteristic, we can 
easily obtain information about the order, as we 
show below. The Nakayama automorphism $\nu$ of an 
AS-Gorenstein Hopf algebra provides the twist in 
its rigid dualizing complex, see 
\cite[Sections 0.2, 0.3]{BZ1} for details. It is 
defined up to an inner automorphism. The 
\emph{Nakayama order}, that is the lowest power 
$\ell$ such $\nu^{\ell}$ is inner, is finite in 
the current setting thanks to 
\cite[Corollary 0.6(c)]{BZ1}. Again, more can be 
said about its precise value.

\begin{theorem}
\label{xxthm3.8} 
Let $n\geq 1$ and let $H$ be an $n$-step IHOE over 
a field $\Bbbk$ of characteristic $p > 0$.
\begin{enumerate}
\item[(1)] 
The order of the antipode $S$ divides $4 p^{n-1}$.
\item[(2)] 
The Nakayama order of $H$ 
divides $2 p^{n}$.
\end{enumerate}
\end{theorem}  
 
\begin{proof} Note that every invertible element in $H$ 
is in $\Bbbk$, so the Nakayama order of $H$ coincides with the 
(well-defined) order of the Nakayama automorphism $\nu$ of $H$.

\noindent
(1) The proof of part (1) is similar to the proof of 
Proposition \ref{xxpro1.4}(2). We use induction on $n$
to show that the order of $S^4$ divides $p^{n-1}$. 
The assertion holds trivially for $n=1$. Now assume that the 
assertion holds for $(n-1)$-step IHOEs. Let $H$ an 
$n$-step IHOE and write it as $B[x;\sigma,\delta]$
where $B$ is an $(n-1)$-step IHOE. Suppose 
the order of $S^4$ restricted to $B$ is $p^t$ for some 
$t\leq n-2$. 

Let $\eta: H\to \Bbbk$ be the right character of the 
left homological integral of $H$ \cite{LWZ}.
Let $\Xi^l_{\eta}$ [respectively, $\Xi^r_{\eta}$] be the 
corresponding left [respectively, right] winding 
automorphism of $H$. By \cite[Theorem 0.6]{BZ1} and the 
fact that $H$ has no nontrivial inner automorphism,
$S^4=\Xi^l_{\eta}\circ (\Xi^r_{\eta})^{-1}$.   
Since $B$ is a Hopf subalgebra of 
$H$ and $\eta\mid_{B}$ is a character of $B$, 
$\Xi^l_{\eta}$ [respectively, $\Xi^r_{\eta}$] restricted 
to $B$, still denoted by $\Xi^l_{\eta}$ [respectively, 
$\Xi^r_{\eta}$], is a left [respectively, right] winding 
automorphism of $B$. By the induction hypothesis, 
$(S^4)^{p^t}$ is the identity of $B$. It remains to show 
that $(S^4)^{p^{t+1}}$ is the identity when applied to 
$x$. By Proposition \ref{xxpro1.2}(2), we may assume that 
\eqref{E1.4.1} holds, so that
$$\Xi^l_{\eta}(x)=x +\eta(x)+\mu \circ (\eta\otimes 1)(w)=x+s$$
where $s=\eta(x)+\mu \circ (\eta\otimes 1)(w)\in B$. 
Similarly,
$$(\Xi^r_{\eta})^{-1}(x)=x+s'$$
for some $s'\in B$. These imply that 
$$(S^4)^{p^t}(x)=(\Xi^l_{\eta} (\Xi^r_{\eta})^{-1})^{p^t}(x)=x+s''$$
for some $s''\in B$. Note that $(S^4)^{p^t}(s'')=s''$ as 
the order of $S^4$ restricted to $B$ is $p^t$. Then
$$(S^4)^{p^{t+1}}(x)=((S^4)^{p^t})^{p}(x)=x+ps''=x.$$
Therefore $(S^4)^{p^{t+1}}$ is the identity map on $H$ 
as required.

\noindent
(2) By \cite[Theorem 0.3]{BZ1}, the Nakayama automorphism 
of $H$ is of the form $S^2 \circ \Xi^{r}_{\eta}=
\Xi^{r}_{\eta}\circ S^2$. The order of $S^2$ 
divides $2p^{n-1}$ by part (1) and the order of $\Xi^{r}_{\eta}$
divides $p^n$ by Proposition \ref{xxpro1.4}(2). Then the
assertion follows.
\end{proof}

\begin{remark}
\label{xxrem3.9}
Recall that $\eta$ is the right character of the 
left integral of $H$.  With a slightly more careful 
analysis along the above lines, combined also with 
Proposition \ref{xxpro8.1}, one can show that 
the order of $S$ divides $\mathrm{max}\{1, 4 p^{n- t}\},$ 
where
$$t := \mathrm{max}\{2, n_0\} \qquad \textit{ and } 
\qquad n_0 :=\# \{ i\mid \Xi^{l}_{\eta}(w_i)=\Xi^r_{\eta}(w_i)\},$$
where $w_i$ is as in \eqref{E1.2.1}.
\end{remark}

\section{Associated graded algebra of an IHOE}
\label{xxsec4} 
The main result in this section is Theorem \ref{xxthm4.3}. 
In this section, we do not assume that 
${\rm{char}}\; \Bbbk>0$ though we keep using $\Bbbk$ as 
the base field.

\subsection{Construction of a filtration}
\label{xxsec4.1} 
Recall that, by \cite[Theorem 6.9]{Zh}, if $H$ is an 
affine connected Hopf algebra of finite Gel'fand-Kirillov 
dimension $n$ over an algebraically closed field $\Bbbk$ 
of characteristic 0, then the associated graded algebra 
of $H$ with respect to the coradical filtration is a 
commutative polynomial algebra in $n$ variables. This 
result does not extend to positive characteristic - in 
the first place there are nontrivial finite dimensional 
connected Hopf algebras in characteristic $p$, for 
example $\Bbbk [x]/\langle x^p \rangle$; and, second, 
the coradical filtration is typically much coarser in 
this setting. Thus, for instance, there is often an infinite 
dimensional space of primitive elements, as 
in $\Bbbk [x]$ for example. At least for IHOEs there is nevertheless 
an analogue of Zhuang's result, provided one uses a more 
suitable filtration. We begin by describing such a 
filtration and proving its required properties.

Let $H$ be an $n$-step IHOE with notation as in 
Definition \ref{xxdef1.1}. The ordered monomials 
$x_1^{m_1} \cdots x_n^{m_n}$ with 
$m_i \in \mathbb{Z}_{\geq 0}$ constitute a 
$\Bbbk$-basis $\mathcal{B}(H)$ of $H$, which we call its \emph{PBW basis}. 
(This holds true for all iterated Ore extensions, not 
just for Hopf algebras.)
For $\alpha \in H$, call the presentation of $\alpha$ in 
terms of the PBW basis the \emph{PBW expression of $\alpha$}, 
and let 
$$\mathrm{PBWsupp}(\alpha) := \{ \mathbf{m} \in \mathcal{B}(H) \mid \mathbf{m} 
\textit{ occurs in the PBW expression for } \alpha \}.$$
We proceed now to define a degree function 
$\mathrm{deg} : H \longrightarrow \mathbb{Z}_{\geq 0}$.  
Recall from Proposition \ref{xxpro1.2}(4) the elements 
$a_{ji} \in H_{(j-1)}$ for $i,j = 1, \dots , n$ with $i>j$, 
such that
\begin{equation}\notag 
\sigma_i (x_j) = x_j + a_{ji}
\end{equation}
where $a_{ji}$ is given as in \eqref{E1.2.2}.
Moreover we can define elements $c_{ji} \in H_{(i-1)}$ 
such that, for all $i,j = 1, \ldots , n$ with $i > j$,
\begin{equation}\notag
\delta_i(x_j) = c_{ji}.
\end{equation}
Thus the defining relations of $H$ are  
\begin{equation}
\label{E4.0.1}\tag{E4.0.1} 
x_i x_j = x_j x_i + a_{ji} x_i + c_{ji}, \qquad 
(1 \leq j < i \leq n). 
\end{equation}
Recall that $w_i$ is defined in 
\eqref{E1.2.1}; fix an expression
$$w_i=\sum w_{i(1)}\otimes w_{i(2)},$$
where $w_{i(1)},w_{i(2)}$ are in $H_{(i-1)}\cap 
\ker \epsilon$. For each $\ell>i$, let $\chi_{\ell}$
be the character of $H_{(\ell-1)}$ such that
$\sigma_{\ell}=\tau^{r}_{\chi_{\ell}}$ as given in 
Proposition \ref{xxpro1.2}(4). Then by
\eqref{E1.2.2},
\begin{equation}
\notag
a_{i\ell} = \chi_{\ell}(x_i) + 
\sum w_{i(1)}\chi_{\ell}(w_{i(2)}) \in H_{(i-1)}.
\end{equation}
Define the degree $\deg(\lambda) := 0$ for 
$\lambda \in \Bbbk$, and inductively define the degree 
$\deg(x_i) := d_i \in \mathbb{Z}_{> 0}$ and 
$\mathrm{deg}(\alpha)$ for $\alpha \in H_{(i)}$, as 
follows:
\begin{enumerate}
\item[$\bullet$] 
$d_1 = 1.$
\item[$\bullet$] 
Suppose that $i > 1$ and that $d_1, \ldots  , d_{i-1}$ 
have been defined. For each PBW basis monomial 
$\mathbf{m} = x_1^{m_1} x_2^{m_2} \cdots x_{i-1}^{m_{i-1}}$, 
set
\begin{equation}
\label{E4.0.2}\tag{E4.0.2} 
\deg(\mathbf{m}) := \sum_{j=1}^{i-1} m_j d_j. 
\end{equation}
\item[$\bullet$] 
For $\alpha \in H_{(i - 1)}$ define
\begin{equation}
\label{E4.0.3} \tag{E4.0.3}
\deg(\alpha) 
:= \max\{ \deg(\mathbf{m}) \mid 
\mathbf{m} \in \mathrm{PBWsupp}(\alpha) \}. 
\end{equation}
\item[$\bullet$]
Since $\deg$ is defined on $H_{(i-1)}$, let 
$$D(w_i)=\max_{i(1)}\{ \deg(w_{i(1)})\}$$
which is only dependent on $i$. Then we have 
$$D(w_i)\geq \max_{\ell > i} \{ \deg(a_{i\ell})\}.$$
Now set $d_i \in \mathbb{Z}_{> 0}$ such that
\begin{equation}
\label{E4.0.4}\tag{E4.0.4}
d_i> D(w_i)\quad  
( \geq \max_{\ell>i} \{ \deg(a_{i\ell})\}),
\end{equation}
and such that
\begin{equation}
\label{E4.0.5}\tag{E4.0.5}
d_i \geq  \max_{j<i} \{ \deg(c_{ji})\}.
\end{equation}
\end{enumerate}
Finally, after the above steps are completed up 
to $i = n$, extend \eqref{E4.0.2} 
[resp. \eqref{E4.0.3}] to \emph{all} PBW basis 
monomials [resp. every element] of $H$. Clearly 
this procedure yields a well-defined degree 
for all $\alpha \in H$. We define a filtration 
$\mathcal{C}:=\{\mathcal {C}_i\}_{i\geq 0}$ by setting,
for all $i\geq 0$,
$$\mathcal{C}_i := \{ \alpha \in H \mid \deg(\alpha) 
\leq i \}.$$ 
It is clear that $\mathcal{C}_0 = \Bbbk$. 

\begin{lemma}
\label{xxlem4.1} Keep the above notation.
\begin{enumerate}
\item[(1)] 
For $1 \leq j < i \leq n$,
$$ x_i x_j = x_j x_i + (\textit{lower degree terms}). $$
\item[(2)] 
For $1 \leq j < i \leq n$,
$$ \deg(x_i x_j) = \deg(x_j x_i). $$
\item[(3)] 
For all monomials 
$\mathbf{m} = x_{i_1}^{t_1}x_{i_2}^{t_2} \cdots 
x_{i_m}^{t_m}$, where $i_{\ell} \in \{1, \ldots , n \}$ 
and $t_{\ell} \in \mathbb{Z}_{\geq 0}$ for all 
$\ell = 1, \ldots , m$,
\begin{equation} 
\label{E4.1.1}\tag{E4.1.1} 
\deg(\mathbf{m}) = \sum_{\ell = 1}^m t_{\ell}d_{i_{\ell}}.
\end{equation} 
\end{enumerate}
\end{lemma}

\begin{proof} (1) 
Let $j < i$. In the relation \eqref{E4.0.1} 
$a_{ji} \in H_{(j-1)}$ and $c_{ji} \in H_{(i-1)}$. Thus
\begin{eqnarray*} 
\deg(a_{ji} x_i + c_{ji}) &=& \max\{\deg(a_{ji} x_i), 
\deg( c_{ji}) \}\\
&<& d_j + d_i \\
&=& \deg(x_j x_i),
\end{eqnarray*}
where the first equality follows from \eqref{E4.0.3} 
and the fact that there is empty intersection between 
the PBW expressions for $a_{ji}x_i$ and $c_{ji}$, so 
that no cancellation can occur, and the inequality 
follows from \eqref{E4.0.4} and \eqref{E4.0.5}, since 
$a_{ji}x_i$ is already in PBW form.

\noindent (2) 
This follows from part (1) and the definition of the 
degree of elements of $H$.

\noindent (3) 
Define the \emph{weight} of a not necessarily 
PBW-ordered monomial $\mathbf{m} = 
x_{i_1}^{t_1} \cdots x_{i_m}^{t_m}$ to be
$$ \mathrm{wt}(\mathbf{m}) := \sum_{\ell = 1}^m t_{\ell} 
d_{i_{\ell}}. $$
Suppose that assertion (3) is false, and choose a 
counterexample $\mathbf{m}$ of minimal weight, say 
$\mathrm{wt}(\mathbf{m}) = d > 0$. Then $\mathbf{m}$ 
cannot be an ordered monomial by the definition 
\eqref{E4.0.2}. So amongst counterexamples to assertion (3) 
of weight $d$, choose $\mathbf{m}$ to be one with the 
minimal number $\omega$ of \emph{badly ordered pairs} 
- that is, pairs of generators occurring in $\mathbf{m}$ 
as $\mathbf{m} = \cdots x_r \cdots x_s \cdots $ with 
$r > s$. Clearly, there must be an adjacent bad pair 
in $\mathbf{m}$. That is, there exist $s,r$ with 
$1 \leq s < r \leq n$, with
\begin{eqnarray*}
\mathbf{m} &=& x_{i_1}\cdots x_r x_s \cdots x_{i_m}\\
&=& x_{i_1} \cdots (x_s x_r + a_{sr} x_r + c_{sr}) \cdots x_{i_m}\\
&=& (x_{i_1} \cdots x_s x_r \cdots x_{i_m}) 
     + (x_{i_1} \cdots  a_{sr} x_r \cdots x_{i_m}) \\
& & \qquad \qquad + (x_{i_1} \cdots  c_{sr} \cdots x_{i_m}).
\end{eqnarray*}
Here, the first monomial on the right side has fewer than 
$\omega$ badly ordered pairs, so assertion (3) is true 
for it by choice of $\mathbf{m}$. The second and third 
brackets on the right consist of monomials of weight 
strictly less than $d$, by \eqref{E4.0.4} and 
\eqref{E4.0.5}. So, again by choice of $\mathbf{m}$, 
assertion (3) holds for all the monomials in the second 
and third brackets on the right; in particular, their 
degree is strictly less than the degree of 
$x_{i_1} \cdots x_s x_r \cdots x_{i_m}$. Thus 
$$ \deg(\mathbf{m}) = \deg(x_{i_1} \cdots x_s x_r 
\cdots x_{i_m}) = \sum_{\ell = 1}^m t_{\ell} d_{i_{\ell}}, $$
and $\mathbf{m}$ is not a counterexample. This proves (3).
\end{proof}

\begin{lemma}
\label{xxlem4.2}
Continue with the above notation.
\begin{enumerate}
\item[(1)] 
The filtration $\mathcal{C}$ defined before Lemma 
\ref{xxlem4.1} is an algebra filtration of $H$.
\item[(2)] 
The associated graded algebra 
$\mathrm{gr}_{\mathcal{C}}H$ is a factor ring of 
the commutative polynomial ring 
$\Bbbk[\overline{x}_1, \ldots , \overline{x}_n]$
where $\overline{x_i}$ is the principal symbol of
$x_i$.
\end{enumerate} 
\end{lemma}

\begin{proof} (1) 
We need to prove that, for all 
$r,s \in \mathbb{Z}_{\geq 0}$,
$$ \mathcal{C}_r \mathcal{C}_s \subseteq 
\mathcal{C}_{r + s}. $$
For this it is enough to show that if 
$\mathbf{r}, \mathbf{s}$ are ordered monomials 
in $\mathcal{C}_r$ and $\mathcal{C}_s$ respectively, 
then $\deg(\mathbf{r}\mathbf{s}) \leq r + s.$ 
This is immediate from Lemma \ref{xxlem4.1}(3).

\noindent (2) 
It is clear from the definitions that 
$0 \neq \overline{x}_i \in \mathrm{gr}_{\mathcal{C}} H$ 
for all $i = 1, \ldots , n$. From this and \eqref{E4.1.1},
one sees immediately that $\mathrm{gr}_{\mathcal{C}} H$ 
is generated by $\{\overline{x}_i\}_{i=1}^n$.

Commutativity of $\mathrm{gr}_{\mathcal{C}}H$ follows from 
Lemma \ref{xxlem4.1}(1). Therefore there is an algebra
epimorphism $\phi$ from the commutative polynomial ring 
$\Bbbk[\overline{x}_1, \ldots , \overline{x}_n]$ to
$\mathrm{gr}_{\mathcal{C}}H$. 
\end{proof}

\subsection{The associated graded algebra}
\label{xxsec4.2} 
We can now prove the main result of this section.

\begin{theorem}
\label{xxthm4.3} 
Let $\Bbbk$ be an algebraically closed field, and let 
$H$ be an $n$-step IHOE over $\Bbbk$. Then 
positive integer degrees can be assigned to 
the defining variables of $H$ so that the 
corresponding filtration, constructed from 
this assignment by using the PBW basis as above, 
has associated graded algebra which is the 
commutative polynomial $\Bbbk$-algebra on $n$ 
variables.
\end{theorem}

\begin{proof} Let $H$ be given by Definition 
\ref{xxdef1.1} and define its filtration 
$\mathcal{C}$ by the recipe given in $\S$\ref{xxsec4.1}. 
By Lemma \ref{xxlem4.2}, $\mathrm{gr}_{\mathcal{C}} H$ 
is a factor of the polynomial $\Bbbk$-algebra on 
$n$ generators. On the other hand, 
$\GKdim (H) = n$ by Corollary \ref{xxcor2.4}. 
Since $\mathcal{C}$ is discrete and finite, it 
follows from \cite[Proposition 6.6]{KL} that 
$\mathrm{gr}_{\mathcal{C}} H$ grows at the same 
rate; that is, $\GKdim (\mathrm{gr}_{\mathcal{C}} H) = n$. 
Since proper factors of the polynomial algebra 
on $n$ generators have GKdimension strictly less 
than $n$, the result follows.
\end{proof}

\begin{remarks}
\label{xxrem4.4} 
(1) An alternative proof that 
$\mathrm{gr}_{\mathcal{C}}H$ is a polynomial algebra 
uses the PBW basis directly. Namely, by the definition of
$\mathcal {C}$, $H$ and $\mathrm{gr}_{\mathcal{C}}H$ have 
the same PBW basis, which agrees with the PBW basis 
of $\Bbbk[\overline{x}_1, \ldots , \overline{x}_n]$.
Therefore the map $\phi$ from the proof of Lemma 
\ref{xxlem4.2}(2) is an isomorphism. 

\noindent 
(2) A theorem of Etingof, \cite[Corollary 3.2(ii)]{Et}, 
states that if $\Bbbk$ is a field of characteristic 
$p > 0$ and $A$ is any $\Bbbk$-algebra with 
filtration $\mathcal{C} = \{\mathcal{C}_i : i \geq 0 \}$ 
such that $A$ satisfies a polynomial identity and 
$\mathrm{gr}_{\mathcal{C}} A$ is a commutative 
domain, then the PI-degree of $A$ is a power of $p$. 
In view of Corollary \ref{xxcor2.4} and Theorem 
\ref{xxthm4.3} these hypotheses both are 
satisfied by any IHOE over the field $\Bbbk$, 
yielding a second proof of Theorem \ref{xxthm3.3}.

\noindent 
(3) Etingof asks in \cite[Question 1.1]{Et} whether 
every filtered deformation of an affine commutative 
domain in positive characteristic has to satisfy a 
polynomial identity. Theorem \ref{xxthm4.3} coupled 
with Corollary \ref{xxcor2.4} provide 
some evidence in favour of a positive answer.

\noindent 
(4) Theorem \ref{xxthm4.3} can be restated as: 
every $n$-step IHOE over the field $\Bbbk$ 
is a PBW-deformation of the polynomial ring in $n$ 
variables over $\Bbbk$ in the sense of 
\cite[Section 3]{BeG}. It follows from Zhuang's
theorem \cite[Theorem 6.9]{Zh} that a connected
Hopf algebra over a field $\Bbbk$ of characteristic
zero is a PBW-deformation of the polynomial ring 
in $n$ variables if $n=\GKdim H<\infty$. Given these 
facts, together with the classical PBW theorem for 
enveloping algebras of finite dimensional 
Lie algebras, and the fact that the underlying 
variety of every connected unipotent group of 
dimension $n$ is affine $n$-space over $\Bbbk$, 
\cite[Chap.VII, sec.6, Corollary and Remark 1, p.170]{Se},  
it is natural to ask:

\begin{question}
\label{xxque4.5} 
Is every connected Hopf $\Bbbk$-algebra domain 
of finite Gel'fand-Kirillov dimension a 
PBW-deformation of a polynomial algebra over 
$\Bbbk$?
\end{question}

\noindent 
(5) An $\infty$-step IHOE is defined to be 
$$\lim_{n\to\infty} H_n$$ 
if there is an infinite sequence of $n$-step IHOEs
$$\Bbbk=H_0\subset H_1\subset H_2\subset \cdots 
\subset H_n \subset \cdots$$
such that each $H_n$ satisfies the conditions in 
Definition \ref{xxdef1.1}(2). If $\{d_i\}_{i\geq 1}$
is a strictly increasing sequence of integers defined
by using the process given before Lemma \ref{xxlem4.1},
then we can define a locally finite filtration 
$\mathcal{C}$ of $H$ such that the associated 
graded algebra $\mathrm{gr}_{\mathcal{C}} H$ 
is isomorphic to $\Bbbk[\bar{x}_1,\bar{x}_2,
\cdots,\bar{x}_n,\cdots]$ -- the polynomial 
ring of infinitely many variables.
 
\noindent 
(6) Let $\chi$ be a character of $H$.
Then the right winding automorphism $\tau^{r}_{\chi}$ 
of $H$ preserves the filtration $\mathcal{C}$ 
constructed in Section \ref{xxsec4.1}. It follows 
that the winding automorphism $\tau^r_{\chi}$ induces 
an automorphism of $\mathrm{gr}_{\mathcal{C}}H$, which 
one can show to be the identity map.
\end{remarks}

\section{Classification of $1$- and $2$-step IHOEs 
in positive characteristic}
\label{xxsec5}

For the rest of the paper we would like to take the 
first steps in a project to classify the Hopf 
algebra domains of Gel'fand-Kirillov dimension at most 
two in positive characteristic. In characteristic 0 
considerable progress has already been made towards 
this classification, as we shall briefly recall at 
the start of $\S$\ref{xxsec5.2}. First, we deal with 
the well-known case of Gel'fand-Kirillov dimension one. 

\subsection{Hopf domains of GKdimension one}
\label{xxsec5.1} 
Recall that the only connected algebraic groups 
of dimension 1 over an algebraically closed 
field $\Bbbk$ are the additive and multiplicative 
groups of $\Bbbk$ \cite[Theorem 20.5]{Hum}. It 
is easy to remove the hypothesis of commutativity 
from this result, as follows. Let $\Bbbk^{\times}$
denote $\Bbbk \setminus \{0\}$.

\begin{lemma}
\label{xxlem5.1} 
Let $\Bbbk$ be any algebraically closed field. 
The only affine Hopf $\Bbbk$-algebra domains 
of Gel'fand-Kirillov dimension one are the 
coordinate rings $\Bbbk [X]$ and 
$\Bbbk [X^{\pm 1}]$ of $(\Bbbk, +)$ and 
$(\Bbbk^{\times}, \times)$ respectively. In 
particular, the only $1$-step IHOE over 
$\Bbbk$ is $\Bbbk [X]$ with $X$ primitive.
\end{lemma}

\begin{proof}
By \cite[Corollary 7.8 (a)$\Rightarrow$(b)]{LWZ},
$H$ is commutative. Therefore $H$ is the coordinate 
ring of an affine connected algebraic group over 
$\Bbbk$, and the result follows from \cite[Theorem 20.5]{Hum}.
\end{proof}

By \cite[Proposition 2.1]{GZ1}, in characteristic 0 
the hypothesis that the Hopf algebra is affine can 
be removed from Lemma \ref{xxlem5.1}, at the cost of 
adding the group algebras of non-cyclic subgroups of 
$\mathbb{Q}$ to the list. But the proof of 
\cite[Proposition 2.1]{GZ1} does not work in positive 
characteristic. We refer to a recent survey paper 
\cite{BZ2} for some facts about the prime Hopf algebras
of GKdimension one and two. 

\subsection{Hopf domains of GKdimension two}
\label{xxsec5.2}
When $\Bbbk$ has characteristic 0 the classification 
of affine Hopf $\Bbbk$-algebra domains $H$ of 
Gel'fand-Kirillov dimension 2 was achieved in 
\cite{GZ1,GZ2} with the imposition of the extra 
hypothesis that $\mathrm{Ext}^1_H (\Bbbk,\Bbbk) \neq 0$ 
(or equivalently, that the quantum group $H$ contains 
a non-trivial classical subgroup). Then, in \cite{WZZ1}, 
a family of dimension two $\Bbbk$-affine Hopf PI 
domains is constructed which fail to satisfy the 
extra hypothesis, but it is not yet known whether, 
with the addition of this family, the list is 
complete. Staying with $\Bbbk$ of characteristic 0, 
if one restricts attention to \emph{connected} Hopf 
$\Bbbk$-algebras of finite Gel'fand-Kirillov dimension, 
then they are all affine by 
\cite[Theorem 6.9]{Zh}, and the classification is 
complete up dimension at most four \cite{Zh, WZZ3}. 
In all these classification results in characteristic 
0 the outcome takes a similar form - namely, there 
is a finite number of families in each list, with 
each family being given by a finite set of 
discrete or continuously varying parameters.

Turning now to positive characteristic, and restricting 
attention to $2$-step IHOEs, we find that even in this 
very confined setting the situation is completely different 
from that pertaining in characteristic 0. This is shown 
by the main result of this 
subsection, Proposition \ref{xxpro5.6}, which lists all 
the two-step IHOEs in characteristic $p$. Together with 
Proposition \ref{xxpro5.11} in $\S$\ref{xxsec5.3}, 
which describes the isomorphisms and automorphisms 
between these algebras, this classifies 2-step 
IHOEs in positive characteristic. It transpires that 
their description entails an infinite dimensional space 
of parameters. The contrast with characteristic 0 could 
not be starker - when $\Bbbk$ has characteristic 0, a 
result of Zhuang \cite[Proposition 7.4(III)]{Zh} shows 
that there are only two connected Hopf $\Bbbk$-algebras 
of Gel'fand-Kirillov dimension 2, namely the enveloping 
algebras of the two 2-dimensional Lie algebras, both of 
which are obviously IHOEs. Conversely, by 
\cite[Theorem 1.3]{BOZZ}, every IHOE is connected, so 
that \cite[Proposition 7.4(III)]{Zh} provides a 
list of the (two) 2-step IHOEs in characteristic 0.

To understand 2-step IHOEs $\Bbbk[X_1][X_2;\sigma,\delta]$ 
in positive characteristic 
we employ the primitive cohomology studied in \cite{WZZ2}, 
from which we first recall some definitions. Let 
$(C,\Delta)$ be a coalgebra with a fixed grouplike element 
$1_C$. Let $T(C)$ be the tensor algebra over the vector 
space $C$ in cohomological degree 1 with differential 
$\partial$ determined by
\begin{equation}
\label{E5.1.1}\tag{E5.1.1}
\partial (x)=-1_C\otimes x+\Delta(x)-x\otimes 1_C
\in C\otimes C
\end{equation}
for all $x\in C$. Then $\partial$ can uniquely be extended 
to a derivation of $T(C)$ such that $(T(C),\partial)$ is a 
differential graded algebra. Let $B^i(C)$ be the image of 
$\partial^{n-1}: 
C^{\otimes (n-1)} \to C^{\otimes n}$ and $Z^n(C)$ be the 
kernel of $\partial^n: C^{\otimes n}\to C^{\otimes(n+1)}$.  
By \cite[Definition 1.2(1)]{WZZ2},
the $n$th primitive cohomology of $C$ (associated to $g=h=1_C$)
is defined to be 
$$\PP^n_{1_C,1_C}(C)=H^n(T(C),\partial)
=\ker \partial^n/\im \partial^{n-1} = Z^n(C)/B^n(C).$$

If $C$ is 
${\mathbb N}$-graded locally finite, then $T(C)$ is 
${\mathbb Z}^2$-graded and locally finite. As a consequence, 
each $\PP^n_{1_C,1_C}(C)$ is ${\mathbb N}$-graded and 
locally finite. We will use following lemma in the 
computation of primitive cohomology.

\begin{lemma}
\label{xxlem5.2} 
Let $C$ be a connected ${\mathbb N}$-graded
coalgebra and let $A$ be the graded dual algebra of $C$, 
namely, $A_i=\Hom_{\Bbbk}(C_i,\Bbbk)$ for all $i\geq 0$. Then 
$A$ is a connected graded algebra with trivial graded module
$\Bbbk$ and 
$$\dim \PP^n_{1_C,1_C}(C)_i=\dim \Ext^n_{A}(\Bbbk,\Bbbk)_i$$
for all $n$ and $i$.
\end{lemma}

\begin{proof} 
After we identify $\dim \Ext^n_{A}(\Bbbk,\Bbbk)_i$ with
$\dim \Tor_n^{A}(\Bbbk,\Bbbk)_{-i}$ for all $i$, the assertion
is a consequence of \cite[Lemma 3.6(2)]{WZZ2}.
\end{proof}

For the rest of this section we assume that $\Bbbk$ has
positive characteristic $p$. We are now ready to 
compute the primitive cohomology of the coalgebra 
$\Bbbk[X_1]$ from Lemma \ref{xxlem5.1}. The divided 
power Hopf algebra (of one variable) was introduced 
in \cite{Sw}, also see
\cite[Example 5 in Sect. 4.3]{DNR}. By definition, the 
divided power Hopf algebra ${\mathcal T}$ is a $\Bbbk$-vector 
space with basis $\{t_n\}_{n\geq 0}$ and with its 
bialgebra structure determined by
\begin{equation}
\label{E5.2.1}\tag{E5.2.1}
\Delta (t_n)=\sum_{i=0}^{n} t_i\otimes t_{n-i} ,
\quad
\epsilon(t_n)=\delta_{0,n}
\quad {\textit{and}}\quad
t_n t_m={n+m\choose n} t_{n+m}
\end{equation}
for all $n,m\geq 0$. By comparing the structure 
coefficients of the multiplications and comultiplications 
of ${\mathcal T}$ and $\Bbbk[X_1]$ (as in Lemma 
\ref{xxlem5.1}) respectively, one can easily see that 
${\mathcal T}$ is the graded $\Bbbk$-linear dual of 
$\Bbbk[X_1]$, giving part (1) of Proposition \ref{xxpro5.4}.

\begin{convention}
\label{xxcon5.3}
Since we are using various different algebra and/or 
coalgebra structures on the same or similar spaces, 
it is convenient to fix some notation.
\begin{enumerate}
\item[(1)]
Let $A$ denote the algebra ${\mathcal T}$ obtained by forgetting 
the coalgebra structure, namely, $A=\bigoplus_{d\geq 0} \Bbbk t_d$ 
is the divided power algebra of one variable with multiplication 
determined by $t_d t_e={d+e\choose d} t_{d+e}$ for all $d,e\geq 0$.
\item[(2)]
Let $C$ be the graded coalgebra $\Bbbk[X_1]$ given in Lemma 
\ref{xxlem5.1} by forgetting its algebra structure.
\end{enumerate}
\end{convention}

\begin{proposition}
\label{xxpro5.4} Retain the above notation.
\begin{enumerate}
\item[(1)]
$A$ is isomorphic to the graded dual algebra of the coalgebra 
$C$. 
\item[(2)]
As a graded algebra $A$ is generated by $\{t_{p^s}\mid s\geq 0\}$
subject to the relations 
$$(t_{p^s})^p=0, \qquad {\text{and}}\qquad 
t_{p^s} t_{p^t}=t_{p^t} t_{p^s}$$
for all $s<t$. As a consequence,
$$\dim \Ext^1_A(\Bbbk, \Bbbk)_i=\begin{cases} 1 & i= p^s 
\;{\text{ for all $s\geq 0$}}\\
0& {\text{otherwise}}\end{cases}$$
and
$$\dim \Ext^2_A(\Bbbk, \Bbbk)_i=\begin{cases} 1 & i= p^{s+1} 
\;{\text{ for all $s\geq 0$}}\\
1 & i= p^{s}+p^{t} \;{\text{ for all $0\leq s<t$}}\\
0& {\text{otherwise.}}\end{cases}$$
\item[(3)]
Consider $C$ as a graded coalgebra with $1_C$ being the 
identity of $\Bbbk[X_1]$. Then 
$$\dim \PP^1_{1_C, 1_C}(C)_i=\begin{cases} 1 & i= p^s 
\;{\text{ for all $s\geq 0$}}\\
0& {\text{otherwise}}\end{cases}$$
and
$$\dim \PP^2_{1_C, 1_C}(C)_i=\begin{cases} 1 & i= p^{s+1} 
\;{\text{ for all $s\geq 0$}}\\
1 & i= p^{s}+p^{t} \;{\text{ for all $0\leq s<t$}}\\
0& {\text{otherwise.}}\end{cases}$$
\item[(4)]
The following elements in $C\otimes C$ generate a $\Bbbk$-linear basis of
$\PP^2_{1_C, 1_C}(C)$:
\begin{equation}
\label{E5.4.1}\tag{E5.4.1}
Z_s:=\sum_{i=1}^{p-1} \frac{(p-1)!}{i! (p-i)!} 
(X_1^{p^s})^{i}\otimes (X_1^{p^s})^{p-i}
\end{equation}
for each $s\geq 0$, and
\begin{equation}
\label{E5.4.2}\tag{E5.4.2}
Y_{s,t}:=X_1^{p^s}\otimes X_1^{p^t}- X_1^{p^t}\otimes X_1^{p^s}
\end{equation}
for all $s<t$.
\end{enumerate}
\end{proposition}

\begin{proof} 
(2) The assertion follow from the fact that $\Ext^1_A(\Bbbk,\Bbbk)$ can 
be identified with a minimal set of generators and $\Ext^2_A(\Bbbk,\Bbbk)$ 
can be identified with a minimal set of relations.

\noindent (3) This follows from Lemma \ref{xxlem5.2} and parts (1,2).

\noindent (4) By part (3), it suffices to show that elements 
$Z_s$ and $Y_{s,t}$ are in $Z^2(C)$ but not in $B^2(C)$, which 
can be verified by some straightforward computations.
\end{proof}

In the next lemma $Z_s$ and $Y_{s,t}$ are as given in 
\eqref{E5.4.1}-\eqref{E5.4.2}.

\begin{lemma}
\label{xxlem5.5}
Let $H$ be a 2-step IHOE generated by $X_1$ and $X_2$ as in 
Definition \ref{xxdef1.1}(2); that is, 
$H=\Bbbk[X_1][X_2;\sigma,\delta]$. Retain the notation 
introduced in Proposition \ref{xxpro5.4}(4). Then the 
following hold.
\begin{enumerate}
\item[(1)]
$$\Delta(X_1)=X_1\otimes 1+1\otimes X_1$$
and
$$\Delta(X_2)=X_2\otimes 1+1\otimes X_2+ w$$
where
\begin{equation}
\label{E5.5.1}\tag{E5.5.1}
w=\sum_{s\geq 0} b_s Z_s+\sum_{s<t} c_{s,t} Y_{s,t}
\end{equation}
for some scalars $b_s, c_{s,t}\in \Bbbk$.
\item[(2)]
$\sigma(X_1)=X_1+\chi$ where $\chi\in \Bbbk$.
\item[(3)]
$\delta$ is a $\sigma$-derivation of $\Bbbk[X_1]$ such that
$$\Delta(\delta(X_1))
=\delta(X_1)\otimes 1+1\otimes \delta(X_1)-\chi w$$
where $w$ is given in \eqref{E5.5.1}.
\item[(4)]
$\chi w=0$. 
\item[(5)]
$\delta(X_1)=\sum_{s\geq 0} d_s X_1^{p^s}$
for some $d_s\in \Bbbk$.
\end{enumerate}
\end{lemma}

\begin{proof} Recall that $C$ is the coalgebra $\Bbbk[X_1]$ and 
$1_C$ is the identity $1$ in $\Bbbk[X_1]$.

\noindent 
(1) Let $w$ be $\partial(X_2):=\Delta(X_2)-(X_2\otimes 1+1\otimes X_2)$.
By \cite[Theorem 1.3(iii)]{Hua} or \cite[Theorem 2.4(i)(f)]{BOZZ}, 
$w\in \Bbbk[X_1]\otimes \Bbbk[X_1]$ and 
$$w\otimes 1+(\Delta\otimes \mathrm{Id})(w)
=1\otimes w+(\mathrm{Id}\otimes \Delta)(w).$$
This means that $w$ is a 2-cocycle in $Z^2(C)$. Up to a change 
of variable $X_2$, one can assume that $w$ is an element in
$\PP^2_{1_C,1_C}(C)$. By Proposition \ref{xxpro5.4}(3,4), we 
can assume that $w$ is of the form \eqref{E5.5.1}.

\noindent 
(2) This follows from \cite[Theorem 2.4(i)(d)]{BOZZ}.

\noindent 
(3) This follows from the fact that $\Bbbk[X_1]$ is 
commutative and by taking $r = X_1$ in 
\cite[Theorem 2.4(i)(d)]{BOZZ}.

\noindent 
(4) Let $\partial$ be the differential of the differential 
graded algebra $T(C)$ as in \eqref{E5.1.1}, so $\partial$ 
is determined by
$$\partial(f)=\Delta(f)-f\otimes 1-1\otimes f$$ 
for all $f\in C=\Bbbk[X_1]$. By part (3), $\partial(\delta(X_1))
= -\chi w$. This means that $\chi w=0$ in $\PP^2_{1_C,1_C}(C)$.
By \eqref{E5.5.1} and the fact that $\{Z_s\}_{s\geq 0}\cup 
\{Y_{s,t}\}_{(s<t)}$ form a basis of  $\PP^2_{1_C,1_C}(C)$, 
it follows that $\chi w=0$.

\noindent 
(5) By parts (3) and (4), $\delta(X_1)$ is primitive. It is 
well-known that every primitive element in $\Bbbk[X_1]$ is 
of the form $\sum_{s\geq 0} d_s X_1^{p^s}$ for some 
$d_s\in \Bbbk$. This is also a consequence of Proposition 
\ref{xxpro5.4}(3).
\end{proof}

Recall now the construction of 2-step IHOEs from the 
introduction. Namely, let ${\bf d_s} = \{d_s\}_{s\geq 0}$, 
${\bf b_s} = \{b_s\}_{s\geq 0}$ and 
${\bf c_{s,t}} = \{c_{s,t}\}_{0\leq s<t}$ be sequences of 
scalars in $\Bbbk$ with only finitely many nonzero elements. 
Let $H({\bf d_s,b_s,c_{s,t}})$ denote the Ore extension 
$\Bbbk[X_1][X_2; \mathrm{Id},\delta]$, where 
$$\delta(X_1)=\sum_{s\geq 0} d_s X_1^{p^s}.$$ 
Moreover, the comultiplication $\Delta: H({\bf d_s,b_s,c_{s,t}})
\to H({\bf d_s,b_s,c_{s,t}}) \otimes H({\bf d_s,b_s,c_{s,t}})$ 
is determined by
$$\begin{aligned}
\Delta(X_1)&=X_1\otimes 1+1\otimes X_1,\\
\Delta(X_2)&=X_2\otimes 1+1\otimes X_2+ w,
\end{aligned}$$
where 
\begin{align}
\notag
w&=\sum_{s\geq 0} b_s 
  \left(\sum_{i=1}^{p-1} \frac{(p-1)!}{i! (p-i)!} (X_1^{p^s})^{i}
  \otimes (X_1^{p^s})^{p-i}\right)\\
	\notag
&\qquad \qquad+\sum_{0 \leq s<t} c_{s,t} \left(X_1^{p^s}\otimes 
  X_1^{p^t}- X_1^{p^t}\otimes X_1^{p^s}\right).
\end{align}
Similarly, define maps $\epsilon$ and $S$ from $\{X_1,X_2\}$ 
to $\Bbbk$ and $H({\bf d_s,b_s,c_{s,t}})$ respectively, by
$$\begin{aligned} \epsilon(X_1)&=\epsilon(X_2)=0,\\
S(X_1)&=-X_1,\\
S(X_2)&=-X_2-m(\mathrm{Id} \otimes S)(w).
\end{aligned}$$

Now we are ready to prove Theorem \ref{xxthm0.3}(1),(2).

\begin{proposition}
\label{xxpro5.6} Retain the above definitions and notation.
\begin{enumerate}
\item[(1)] 
The definitions above of $\Delta$, $\epsilon$ and $S$ 
extend uniquely to $H({\bf d_s,b_s,c_{s,t}})$ so that it 
is a Hopf algebra.
\item[(2)] 
Let $H:=\Bbbk[X_1][X_2;\sigma,\delta]$ be a 2-step IHOE 
generated by $X_1$ and $X_2$ as in Definition 
\ref{xxdef1.1}(2). Then $H$ is isomorphic to 
$H({\bf d_s,b_s,c_{s,t}})$, for a suitable choice of scalars. 
\end{enumerate}
\end{proposition}

\begin{proof} (1) This follows from \cite[Theorem 2.4(ii)]{BOZZ}
and an easy computation (similar to the proof of 
Lemma \ref{xxlem5.5}). 

\noindent 
(2) Let $\chi$ be the scalar given in Lemma 
\ref{xxlem5.5}(2). Up to a change of variable, we can assume 
that $\chi$ is either 0 or 1. 

{\bf Case 1:} $\chi=0$.  The assertion follows from Lemma 
\ref{xxlem5.5}(1,5).

{\bf Case 2:} $\chi=1$. In this case 
$\sigma(X_1)=X_1+1$ by Lemma \ref{xxlem5.5}(2). By Lemma 
\ref{xxlem5.5}(4), $w=0$. By Lemma \ref{xxlem5.5}(1),
both $X_1$ and $X_2$ are primitive. Since $\sigma(X_1)=
X_1+1$, we have a relation
$$X_2X_1=X_1X_2+X_2 + \sum_{s\geq 0} d_s X_1^{p^s}.$$
Replacing $X_2$ by the new primitive generator 
$\widehat{X_2} := X_2 + \sum_{s\geq 0} d_s X_1^{p^s}$, 
the above relation becomes
$$\widehat{X_2}X_1=X_1\widehat{X_2}+\widehat{X_2}.$$
Exchanging $X_1$ and $\widehat{X_2}$ and changing the 
sign of $X_1$, one now sees that
$H$ is isomorphic as a Hopf algebra to $H({\bf d_s, 0,0})$
where ${\bf d_s} = \{d_0=1,\, d_s=0 :s \geq 0\}$. This 
completes the proof.
\end{proof}

We shall see in the next subsection that Proposition 
\ref{xxpro5.6} implies that there is an immense zoo 
of isomorphism classes of 2-step IHOEs in positive 
characteristic. But, in contrast to this plethora, 
the classification up to birational equivalence is 
very simple. In fact, it exactly parallels the story 
in characteristic 0, where - in the light of 
Zhuang's result \cite[Proposition 7.4(III)]{Zh}, 
there are 2 birational equivalence classes, with 
quotient division rings the field of rational functions 
$\Bbbk(X_1, X_2)$ and the quotient division ring 
$Q(A_1(\Bbbk))$ of the first Weyl algebra in the 
commutative and noncommutative cases respectively.

\begin{corollary}
\label{xxcor5.7} 
Let $\Bbbk$ be algebraically closed of positive 
characteristic and let $H$ be a 2-step IHOE over 
$\Bbbk$. If $H$ is not commutative, that is if 
$\mathbf{d_s} \neq \mathbf{0}$, then the quotient 
division ring of $H$ is isomorphic to $Q(A_1(\Bbbk))$, 
the first Weyl skew field over $\Bbbk$.
\end{corollary}

\begin{proof} By Proposition \ref{xxpro5.6}(2) 
$H \equiv H({\bf d_s,b_s,c_{s,t}})$ for some choice 
of the parameters, with $\mathbf{d_s} \neq \mathbf{0}$ 
since $H$ is by assumption not commutative. Thus 
$H = \Bbbk \langle X_1, X_2 \rangle$ with 
$$ [X_2, X_1] = \delta(X_1) = \sum_{s\geq 0} d_s X_1^{p^s}.$$
The quotient division ring $Q$ of $H$ is therefore 
generated by $X_2(\delta(X_1))^{-1}$ and $X_1$, and 
these generators satisfy the defining relation of the 
Weyl skew field, as required.
\end{proof}

\subsection{Classification of 2-step IHOEs: 
isomorphisms and automorphisms}
\label{xxsec5.3}

This subsection has two interconnected purposes: we 
complete the classification begun in Proposition 
\ref{xxpro5.6} by determining when any two of the 
algebras listed there are isomorphic; and in so doing 
we describe all the automorphisms of these Hopf 
algebras. The proof of the main result requires 
three preliminary lemmas.

\begin{lemma}
\label{xxlem5.8}
Let $H = H({\bf 0, b_s, c_{s,t}})$ be a commutative 
2-step IHOE as described in Proposition \ref{xxpro5.6}, 
so
$$\begin{aligned}
\Delta(X_1)&=X_1\otimes 1+1\otimes X_1,\\
\Delta(X_2)&=X_2\otimes 1+1\otimes X_2+ w
\end{aligned}
$$ 
where $w$ is given  in \eqref{E0.2.2}. Let $P(H)$ 
denote its subspace of primitive elements.
\begin{enumerate}
\item[(1)]
If $w=0$, that is if ${\bf b_s} = {\bf c_{s,t}} = {\bf 0}$, 
then $P(H)$ has $\Bbbk$-basis 
$\{X_1^{p^i}, \, X_2^{p^j} : i,j \geq 0 \}$, so that 
$\Bbbk \langle P(H) \rangle = H$.
\item[(2)]
If $w\neq 0$, then $P(H)$ has $\Bbbk$-basis 
$\{X_1^{p^i}: i \geq 0 \}$, and 
$\Bbbk\langle P(H) \rangle = \Bbbk[X_1]$.
\end{enumerate} 
\end{lemma}

\begin{proof}
Let $f:=\sum_{i=0}^n f_i X_2^i \in P(H)$
where $f_i\in \Bbbk[X_1]$. 
If $n=0$, then $f=f_0\in \Bbbk[X_1]$. It then follows from a 
direct computation that $f$ is of the form $\sum_{s\geq 0} 
\alpha_s X_1^{p^s}$ for some finite sequence of scalars 
$\alpha_s$. (This fact is also well-known.)

Next assume that the $X_2$-degree of $f$ is $n$, with 
$n\geq 1$, and consider the equation
\begin{equation}
\label{E5.8.1}\tag{E5.8.1} 
f \otimes 1 + 1 \otimes f = \Delta (f) 
= \sum_{i=0}^n \Delta(f_i)\Delta (X_2)^i. 
\end{equation}
For $i=p^s$, then using commutativity of $H$ and the 
fact that ${\rm{char}}\;\Bbbk=p$,
\begin{equation}
\label{E5.8.2}\tag{E5.8.2}
\Delta(X_2)^{p^s}
=X_2^{p^s}\otimes 1+1\otimes X_2^{p^s}+w^{p^s},
\end{equation}
where $w^{p^s} \in \Bbbk [X_1] \otimes \Bbbk [X_1]$. 

On the other hand, if $i\leq n$ and $i\neq p^s$, then the 
expansion of $\Delta(X_2^i)$ has at least one nonzero 
term of the form $c X_2^{j}\otimes X_2^{i-j}$ where 
$0<j<i$ and $0\neq c\in \Bbbk$. Suppose $f_i\neq 0$ for 
such an integer $i$. Since $H\otimes H$ is a domain, this 
implies that the expansion of $\Delta(f)$ has a nonzero 
term of the form $\Delta(f_i) c X_2^{j}\otimes X_2^{i-j}$ 
where $0<j<i$ and $0\neq c\in \Bbbk$. The total degree 
in $X_2$ of this term is $i$, so it cannot cancel with 
any other term in $f\otimes 1+1\otimes f (=\Delta(f))$. 
This contradicts the hypothesis that $f$ is primitive. 
Therefore $f=f_0 +\sum_{s\geq 0} g_{s} X_2^{p^s}$ for 
some $f_0, g_s\in \Bbbk[X_1]$. Since $f$ is primitive, 
\eqref{E5.8.1} and \eqref{E5.8.2} imply that
$$f\otimes 1+1\otimes f=\Delta(f_0)+\sum_{s\geq 0}
\Delta(g_s) 
[X_2^{p^s}\otimes 1+1\otimes X_2^{p^s}+w^{p^s}],$$
and hence $\Delta(g_s)=1\otimes g_s=g_s\otimes 1$
for all $s$. The counital  Hopf algebra axiom 
now forces $\beta_s:=g_s\in \Bbbk$ for all $s$. 

\noindent 
(1) If $w=0$, then $\sum_{s\geq 0} \beta_{s} X_2^{p^s}$
is clearly primitive. Hence $f_0$ is primitive and
$f_0=\sum_{s\geq 0} \alpha_s X_1^{p^s}$ for some
$\alpha_s\in \Bbbk$, and the claims follow.

\noindent (2) If $w\neq 0$ then it is of the form 
\eqref{E0.2.2}. Since ${\rm{char}}\; \Bbbk=p$,
for every $s\geq 0$, $w^{p^s}$ is also of the form of 
\eqref{E0.2.2}, but of higher $X_1$-degree. If $\beta_s\neq 0$ 
for some $s$, by counting the $X_1$-degree one sees that
$\sum_{s\geq 0} \beta_s w^{p^s}$ is another nonzero element of 
the form \eqref{E0.2.2}. In particular, the class of 
$\sum_{s\geq 0} \beta_s w^{p^s}$ in $\PP^2_{1, 1}(\Bbbk[X_1])$ 
is nonzero, by Proposition \ref{xxpro5.4}(4).

Recall that $f=f_0+\sum_{i\geq 1} f_i X_2^{i}
=f_0+\sum_{s\geq 0} \beta_s X_2^{p^s}$. By 
\eqref{E5.8.2},
$$\begin{aligned}
0&=\Delta(f)-1\otimes f-f\otimes 1\\
&=[\Delta(f_0)-1\otimes f_0-f_0\otimes 1]
+\sum_{s\geq 0} \beta_s w^{p^s}\\
&=\partial(f_0)+\sum_{s\geq 0} \beta_s w^{p^s}.
\end{aligned}
$$
This implies that the class 
$\sum_{s\geq 0} \beta_s w^{p^s}$ in 
$\PP^2_{1, 1}(\Bbbk[X_1])$ is equal to the class 
of $-\partial(f_0)$, which is zero by definition. 
This yields a contradiction. Therefore all 
$\beta_s=0$ and the result follows.
\end{proof}

\begin{lemma}
\label{xxlem5.9}
Let $H = H({\bf d_s,b_s,c_{s,t}})$  be a 
noncommutative 2-step IHOE as in Proposition 
\ref{xxpro5.6}. That is, $\delta(X_1)\neq 0$, 
equivalently ${\bf d_s} \neq {\bf 0}$. 
\begin{enumerate}
\item[(1)]
The commutator ideal $[H,H]$ is a nonzero 
Hopf ideal of $H$. 
\item[(2)]
Let $K$ denote the Hopf algebra $\Bbbk[t]$, so $t$ 
is primitive. Suppose $\phi: H\to K$ is a Hopf 
algebra epimorphism. Then $\ker \phi$ is the 
principal ideal generated by $X_1$.
\item[(3)]
For any Hopf algebra epimorphism $\phi: H\to K$,
$H^{co\; K}$ is the Hopf subalgebra $\Bbbk[X_1]$ 
of $H$.
\end{enumerate}
\end{lemma}

\begin{proof}
(1) It is well-known that the commutator ideal 
is a Hopf ideal \cite[Lemma 3.7]{GZ1}. Since $H$ 
is noncommutative, $[H,H]\neq 0$.

\noindent 
(2) Let $f(X_1)$ denote the polynomial 
$\delta(X_1)\in \Bbbk[X_1]$.
Then $f(X_1)$ is a nonzero element in $[H,H]$. 
Since $\mathrm{im}\phi$ is commutative, $[H,H]$ is a
subspace of $\ker \phi$, so $\phi(f(X_1))=0$. We claim
that $\phi(X_1)=0$. Suppose not, and let $y=\phi(X_1)$. Since
$X_1$ is primitive, so is $y$. Then 
$0 \neq y = \sum_{s\geq 0} \alpha_s t^{p^s}$
for a finite sequence of scalars $\alpha_s$, so that
$$\phi(f(X_1))=f(\phi(X_1))=f(y)=
f(\sum_{s\geq 0} \alpha_s t^{p^s})$$
which is nonzero. This yields a contradiction. 
Thus $y=0$ as required.

\noindent 
(3) As 
$$ (\mathrm{Id} \otimes \phi)\circ \Delta (X_1)
 = X_1 \otimes \phi(1),$$
$X_1 \in H^{co\; K}$ and so 
$\Bbbk [X_1] \subseteq H^{co \; K}$. To prove the 
reverse inclusion, let $g(X_1,X_2)$ be any element 
in $H^{co\; K}$. Since $\phi(X_1)=0$ by part (2) 
and $\phi$ is an epimorphism, $\phi(X_2)$ is an 
algebra generator of $K$. Now
$$\begin{aligned}
g(X_1\otimes 1, X_2\otimes 1)&=g(X_1,X_2)\otimes 1\\
&=(\mathrm{Id} \otimes \phi) \Delta g(X_1,X_2)\\
&=g((\mathrm{Id} \otimes \phi)\Delta(X_1),(\mathrm{Id} \otimes \phi) \Delta (X_2))\\
&=g(X_1\otimes 1, X_2\otimes 1+1\otimes \phi(X_2)-\sum_{t>0} c_{0,t} X_1^{p^t}\otimes 1).
\end{aligned}
$$
This implies that $X_2$-degree of $g$ is zero, 
so $g\in \Bbbk[X_1]$ as required.
\end{proof}

\begin{lemma}
\label{xxlem5.10}
Let $H = \Bbbk \langle X_1,X_2 \rangle$ and 
$H'=  \Bbbk \langle X_1',X_2' \rangle$ be 
two 2-step IHOEs as listed in Proposition 
\ref{xxpro5.6}. Let $\phi: H \to H'$ 
be a Hopf algebra isomorphism such that 
$\phi(\Bbbk [X_1]) = \Bbbk [X_1']$. Then 
$\phi(X_2)=c X_2'+ v(X_1')$ for some 
$0\neq c\in \Bbbk$ and $v(X_1')\in \Bbbk[X_1']$.
\end{lemma}

\begin{proof} We will not need any Hopf algebra structure
for this proof. So up to a change of variable, we can assume
that $\phi(X_1)=X_1'$.

Let $a$ be the $X_2'$-degree of $\phi(X_2)$
and $b$ be the $X_2$-degree of $\phi^{-1}(X_2')$.
Since $\phi(X_1)=X_1'$ and both
$H$ and $H'$ are Ore extensions, it is easy to see 
that $ab=1$. This forces $a=b=1$.
Write $\phi(X_2)=c(X_1') X_2'+v(X_1')$ and 
$\phi^{-1}(X_2')=d(X_1) X_2+ u(X_1)$. Then $c(X_1')
d(X_1')=1$, so $0 \neq c:=c(X_1') \in \Bbbk$ is a 
nonzero scalar in $\Bbbk$.
\end{proof}

We can now describe the Hopf algebra isomorphisms and 
automorphisms between 2-step IHOEs. In what follows 
we shall use the adjective {\it trivial} to specify 
the IHOE $H({\bf 0,0,0})$ - in other words, the 
\emph{trivial} 2-step IHOE is the coordinate ring of 
$(\Bbbk,+) \times (\Bbbk,+)$. To describe 
$\mathrm{Aut}(H({\bf 0,0,0}))$, let $F$ denote the 
automorphism of $(\Bbbk, +)$ which maps  
$\lambda \in \Bbbk$ to $\lambda^p$. Thus $F$ and 
$\Bbbk \setminus \{0\}$ are in $\mathrm{Aut}((\Bbbk, +))$, 
where elements of $\Bbbk \setminus \{0\}$ act by left multiplication, 
and together they generate in $\mathrm{End}((\Bbbk,+))$ 
the skew polynomial subalgebra $A:=\Bbbk [F ; \sigma]$, 
where $\sigma (\lambda) = \lambda^{-p}$ for 
$\lambda \in \Bbbk$.

In the next proposition, let $H$ and 
$\Bbbk\langle X_1,X_2\rangle$ denote 
$H({\bf d_s,b_s,c_{s,t}})$ and let $H'$ and 
$\Bbbk\langle X_1',X_2'\rangle$ denote 
$H({\bf d'_s,b'_s,c'_{s,t}})$ as listed in 
Proposition \ref{xxpro5.6}.

\begin{proposition}
\label{xxpro5.11}
Retain the above notation.
\begin{enumerate}
\item[(1)]
Let $\phi: H\to H'$ be a Hopf algebra isomorphism.
Then there are nonzero scalars $\alpha,\beta$ in $\Bbbk$ 
such that, for all $s \geq 0$, $s<t$,
\begin{equation}
\label{E5.11.1}\tag{E5.11.1}
d'_s=d_s \alpha^{p^s-1}\beta^{-1}, \quad 
b'_s=b_s \alpha^{p^{s+1}}\beta^{-1}, \quad
c'_{s,t}=c_{s,t} \alpha^{p^s+p^t}\beta^{-1}.
\end{equation}
Conversely, if $\alpha$ and $\beta$ are nonzero scalars 
such that \eqref{E5.11.1} holds, then there is a Hopf 
algebra isomorphism 
$\phi:H\to H'$ such that
$$\phi(X_1)=\alpha X_1', \quad {\text{and}}\quad 
\phi(X_2)=\beta X_2'.$$
\item[(2)]
Suppose $H$ is not trivial. Then 
every Hopf algebra isomorphism from $H$ to $H'$ has the form
$$\phi(X_1)=\alpha X_1', \quad {\text{and}}\quad 
\phi(X_2)=\beta X_2'+\sum_{s\geq 0} e_s X_1'^{p^s},$$
for an arbitrary finite sequence of scalars 
$\{e_s : s \geq 0 \}$, and nonzero scalars $\alpha$ and 
$\beta$ satisfying \eqref{E5.11.1}. 
\item[(3)]
If $H$ is not trivial, then every Hopf algebra automorphism 
of $H$ is of the form
$$\begin{aligned}
\phi(X_1)&=\alpha X_1, \\
\phi(X_2)&=\beta X_2+\sum_{s\geq 0} e_s X_1^{p^s}
\end{aligned}
$$
where $\{e_s : s \geq 0 \}$ is an arbitrary finite 
sequence of scalars, and $\alpha,\beta$ are nonzero 
scalars satisfying
\begin{equation}
\label{E5.11.2}\tag{E5.11.2}
d_s=d_s \alpha^{p^s-1}\beta^{-1}, \quad 
b_s=b_s \alpha^{p^{s+1}}\beta^{-1}, \quad
c_{s,t}=c_{s,t} \alpha^{p^s+p^t}\beta^{-1};
\end{equation}
equivalently,
\begin{equation}
\notag
\beta=\begin{cases}
\alpha^{p^s-1} &\textit{ if } \; d_s\neq 0,\\
\alpha^{p^{s+1}} &\textit{ if } \;b_s\neq 0,\\
\alpha^{p^s+p^t} &\textit{ if } \; c_{s,t}\neq 0.
\end{cases}
\end{equation}
\item[(4)]
Suppose that $H$ is trivial, so 
$H \cong \mathcal{O}((\Bbbk,+)^2)$. Then, in the 
notation introduced before the proposition, 
$$\mathrm{Aut}(H) \cong GL_2(A).$$
\end{enumerate}
\end{proposition}

\begin{proof}
(1) If both $H$ and $H'$ are trivial, then the assertions
are obvious. Now we assume that $H$ is not trivial. 

If $H$ is noncommutative, then $H'$ is noncommutative, 
and so $H'$ is not trivial. If $H$ is commutative, then 
$w\neq 0$ and, by Lemma \ref{xxlem5.8}, $H'$ is not 
trivial. Thus, in both cases, $H'$ is not trivial. 

By Lemma \ref{xxlem5.8}(2) and Lemma \ref{xxlem5.9}(3),
the Hopf algebra isomorphism $\phi: H\to H'$ maps 
$\Bbbk[X_1]$ to $\Bbbk [X_1']$. In particular, 
$\Bbbk [X_1'] = \Bbbk [\phi(X_1)]$, so 
$\phi(X_1)=\alpha X_1'+a_0$ for a nonzero scalar 
$\alpha$ and a scalar $a_0$. Since $X_1$ is primitive, 
$a_0=0$ and $\phi(X_1)=\alpha X_1'$. By Lemma \ref{xxlem5.10},
$\phi(X_2)=\beta X_2'+ v(X_1')$, where $\beta$ is a 
nonzero scalar and $v(X_1')\in \Bbbk[X_1']$. Applying
$\phi$ to the defining relation of $H$, namely
$$X_2 X_1-X_1X_2=\delta(X_1)=\sum_{s\geq 0} d_s X_1^{p^s},$$
we obtain the following equation in $H'$:
$$\alpha\beta(X_2' X_1'-X_1'X_2')
=\sum_{s\geq 0} d_s \alpha^{p^s} X_1'^{p^s},$$
which must agree with 
$$X_2' X_1'-X_1'X_2'=\delta'(X_1')
=\sum_{s\geq 0} d'_s X_1'^{p^s}.$$
Therefore $d'_s=d_s \alpha^{p^s-1}\beta^{-1}$ for 
all $s$. 

Recall that $\partial$ denotes the map defined in 
\eqref{E5.1.1}. Applying $\phi\otimes \phi$ to 
the following equation in $H\otimes H$,
$$\Delta(X_2)=X_2\otimes 1+1\otimes X_2+w,$$
we obtain an equation in $H'\otimes H'$, namely
$$\beta \Delta(X_2')+\partial(v(X_1'))=\beta
(X_2'\otimes 1+1\otimes X_2') +(\phi\otimes \phi)(w).$$
With the obvious notation $w' := \partial(X_2')$, 
the above equation must agree with 
$$\Delta(X_2')=X_2'\otimes 1+1\otimes X_2' +w'.$$
Using the explicit form of $(\phi\otimes \phi)(w)$, 
one can show that $\partial(v(X_1'))=0$ and 
$w'=\beta^{-1} (\phi\otimes \phi)(w)$. The former 
implies that $v(X_1')\in \Bbbk[X_1']\subseteq H'$ 
is primitive and the latter implies that
\begin{equation}
\notag
b'_s=b_s \alpha^{p^{s+1}}\beta^{-1}, \quad
c'_{s,t}=c_{s,t} \alpha^{p^s+p^t}\beta^{-1}
\end{equation}
for all $s<t$. Therefore \eqref{E5.11.1} holds.

The converse is clear. 

\noindent (2) By the proof of part (1), $\phi(X_1)=\alpha X_1'$
and $\phi(X_2)=\beta X_2'+v(X_1')$, with $v(X_1')$  
primitive. Since every primitive element element
in $\Bbbk[X_1]$ is of the form 
$\sum_{s\geq 0} e_s X_1^{p^s}$, the assertion follows.

\noindent 
(3) This is a special case of part (2).

\noindent 
(4) Suppose that $H = H({\bf 0,0,0})$ is trivial. 
That $\mathrm{Aut}(H) \cong GL_2(A)$ is an 
immediate consequence of the equivalence of 
categories between the category of unipotent 
algebraic groups over $\Bbbk$ which are 
subgroups of $(\Bbbk, +)^r$ for some $r$, and 
the category of finitely generated $A$-modules, 
under which $(k,+)$ corresponds to the free 
$A$-module of rank 1, 
\cite[Theorem 14.46 and Example 14.40]{Mi}.
\end{proof}

\section{Comments and questions}
\label{xxsec6} 

We gather here a number of observations related 
to the classification results in $\S$\ref{xxsec5.3}.

\subsection{Algebraic groups}
\label{xxsec6.1} 
To recast $\S\S$\ref{xxsec5.2},\ref{xxsec5.3} in the language 
of affine algebraic $\Bbbk$-groups, first consider the overuse 
in this context of the word ``connected'': a Hopf algebra $H$ 
is by definition \emph{connected} if its coradical 
$H_0$ is $\Bbbk$, equivalently if it has a unique 
simple comodule \cite[Definition 5.1.5]{Mo1}; an 
affine algebraic $\Bbbk$-group $G$ is \emph{unipotent} 
if and only if its coordinate ring $H=\mathcal{O}(G)$ 
is a connected Hopf algebra \cite[Theorem 14.5]{Mi}; 
and an affine algebraic $\Bbbk$-group $G$ is 
\emph{connected} if and only if it equals $G^{\circ}$, 
its connected component of $1_G$. Equivalently, 
$G$ is connected if and only if $\mathcal{O}(G)$ 
is a domain\footnote{In the algebraic groups 
literature this confusion is sometimes avoided by 
using \emph{coconnected} for the first of these usages.}. 
In characteristic 0 every unipotent group $U$ is connected 
- that is, $\mathcal{O}(U)$ is always a domain; but the 
cyclic group of order $p$ shows that this is false in 
characteristic $p$. 

The key part $(2)\Longrightarrow (1)$ 
of the following fundamental result is due to Lazard 
\cite{La}; a short proof valid in all characteristics 
can be found in \cite[$\S$4]{KP}. For the full 
statement of the theorem, see for example \cite[$\S$8]{Ka}.

\begin{theorem}
\label{xxthm6.1} 
Let $\Bbbk$ be an algebraically closed field of 
arbitrary characteristic, and let $G$ be an affine 
algebraic $\Bbbk$-group and let $n$ be a positive 
integer. Then the following are equivalent.
\begin{enumerate}
\item[(1)] 
$G$ is a connected unipotent group over $\Bbbk$ of dimension $n$.
\item[(2)] 
$\mathcal{O}(G) \cong \Bbbk [X_1, \ldots , X_n]$ as 
$\Bbbk$-algebras.
\item[(3)] 
$\mathcal{O}(G)$ is an affine connected Hopf domain of 
Gel'fand-Kirillov dimension $n$.
\item[(4)] 
$\mathcal{O}(G)$ is an $n$-step IHOE.
\item[(5)] 
$G$ has a subnormal series of length $n$ with factors 
isomorphic to $(\Bbbk,+)$.
\item[(6)] 
$G$ has a central series of length $n$ with factors 
isomorphic to $(\Bbbk,+)$.
\end{enumerate}
\end{theorem}

\subsection{Classification of connected unipotent 
groups of dimension two} 
\label{xxsec6.2}
Revert to our usual hypothesis that the base 
field is $\Bbbk$, algebraically closed of 
characteristic $p > 0$. The coordinate ring of 
every 2-dimensional connected unipotent group 
over $\Bbbk$ is a 2-step IHOE, by 
$(1) \Longleftrightarrow (4)$ of Theorem 
\ref{xxthm6.1}, so the classification of 
$\S$\ref{xxsec5.3} incorporates a classification 
of these groups. Clearly, in the notation of 
$\S$\ref{xxsec5.2}, $H({\bf d_s,b_s,c_{s,t}})$ 
is commutative if and only if ${\bf d_s} = {\bf 0}$, 
so the 2-dimensional connected unipotent 
$\Bbbk$-groups are classified by
$$ H({\bf 0,b_s,c_{s,t}}) / \sim, $$
where the equivalence is provided by the 
isomorphisms of Proposition \ref{xxpro5.11}(1,2). 
A version of the same classification, over 
an arbitrary field, in group-theoretic language, 
is given (``rather formal'' in the words of 
the authors) at \cite[$\S$3.7]{Ka}.
  
It is easy to write down the group $G$ such 
that $H({\bf 0,b_s,c_{s,t}}) \cong \mathcal{O}(G)$. 
Namely, $G = \mathbb{A}^2(\Bbbk)$, and for 
$(a,e)$ and $(f,g)$ in $G$ 
\begin{align} 
(a,e)\ast & (f,g) \; \label{E6.1.1}\tag{E6.1.1}\\
&= \; (a + f, e + g + 
\sum_{s \geq 0}b_s(\sum_{i=1}^{p-1} 
\lambda_i a^{ip^s}e^{(p-i)p^s}) 
+ \sum_{0\leq s<t} c_{s,t}(a^{p^s}e^{p^t} 
- a^{p^t}e^{p^s})).\notag
\end{align}
From this it is easy to see that
\begin{enumerate}
\item[($\bullet$)] 
$G$ is abelian $\Longleftrightarrow$ 
${\bf c_{s,t}} = {\bf 0}$ or $p = 2$.
\item[($\bullet$)] 
$G$ has exponent $p$ $\Longleftrightarrow$ 
${\bf b_s} = {\bf 0}$.
\end{enumerate}
Combining these two statements yields the conclusion 
that the only abelian connected two-dimensional 
unipotent group of exponent $p$ is $(\Bbbk, +)^2$, a 
special case of a result valid for all dimensions, 
\cite[Proposition VII.11]{Se}, \cite[Lemma 1.7.1]{Ka}.

\subsection{Ore extensions of dimension two}
\label{xxsec6.3} 
The Ore extensions of the polynomial ring $\Bbbk[X_1]$ 
are classified in \cite{AVV}, so one can view Proposition 
\ref{xxpro5.6} as determining which of these Ore 
extensions admit Hopf algebra structures. Notice that 
Propositions \ref{xxpro5.6} and \ref{xxpro5.11} together 
show that, for every Ore extension 
$R =\Bbbk [X_1][X_2 ; \sigma, \delta]$, the number of 
distinct Hopf structures which $R$ can carry is either 
$0$ or $\infty$. This is in stark contrast to the 
situation in characteristic 0, as noted in the opening 
sentence of $\S$\ref{xxsec5.2}. 

A second noteworthy feature is this: every such $R$ 
admitting a Hopf algebra structure can be presented
with $\sigma = \mathrm{Id}_{\Bbbk[X_1]}$. 
(The enveloping algebra of the non-abelian 
2-dimensional Lie algebra can be presented in 2 
distinct ways as an Ore extension, only one of 
which can be written as an Ore extension 
with $\sigma = \mathrm{Id}_{\Bbbk[X_1]}$.)
But this does not extend to higher dimensions: 
$U(\mathfrak{sl}(2, \Bbbk))$ is a 3-step IHOE 
\cite[$\S$3.1, Examples (iii)]{BOZZ}, but cannot 
be so presented without using automorphisms.

Another case to consider is HOE 
$\Bbbk[X_1^{\pm 1}][X_2, \sigma, \delta]$ where 
$X_1$ is a grouplike element. Let $C=\Bbbk[X_1^{\pm 1}]$.
By \cite[Theorem 1.3(i)]{Hua},
there is a grouplike element $\alpha=X_1^{a}\in C$ such that
$$\Delta(X_2)=\alpha\otimes X_2+X_2\otimes 1+w$$
where $w\in C\otimes C$. By \cite[Lemma 2.3(1)]{WZZ2},
$w$ is a $(\alpha,1)$-2-cocycle. Since 
$C$ is cosemisimple, it has primitive cohomology dimension 0
and  
$$\PP^2_{\alpha,1}(C)=0.$$
This implies that $w$ is a $(\alpha,1)$-2-coboundary.
Up to a change of variable $X_2$, we can assume that
$w=0$ by \cite[Lemma 2.3(3)]{WZZ2}. By \cite[Theorem 1.3(ii)]{Hua}, 
there is a nonzero $c\in \Bbbk$ such that 
$\sigma(X_1)=c X_1$. Let $\delta(X_1)=\sum_n b_n X_1^n$.
By \cite[Theorem 1.3(iii)]{Hua}, we have
$$\sum_n b_n X_1^n\otimes X_1^n -(\sum_n b_n X_1^n)\otimes X_1
-X_1^{a+1}\otimes (\sum_n b_n X_1^n)=0$$
which implies that
$$\delta(X_1)= b_1 (X_1-X_1^{a+1}).$$
In summary, we have the following family of Hopf algebras:
$$K(a, b,c):=\Bbbk[X_1^{\pm 1}][X_2;\sigma,\delta]$$ 
where $\sigma(X_1)=c X_1$ for some nonzero $c\in \Bbbk$ and 
$\delta(X_1)=b (X_1-X_1^{a+1})$ for some 
$b\in \Bbbk$, $a\in \mathbb{Z}$, and 
the coalgebra structure of $K(a,b,c)$ is 
determined by
$$\Delta(X_1)=X_1\otimes X_1, \epsilon(X_1)=0$$
and
$$\Delta(X_2)=X_1^a \otimes X_2+1\otimes X_2, 
\epsilon(X_2)=0.$$

The following lemma is clear.

\begin{lemma}
\label{xlem6.2}
Retain the above notation. Then $K(a,b,c)$
is commutative if and only if $b=0$ and 
$c=1$. In this case, there is a unique Hopf ideal
$I$ such that $K(a,0,1)/I\cong \Bbbk[X_2]$. 
\end{lemma}

\subsection{Classification of connected algebraic 
groups of dimension two}
\label{xxsec6.4}
The following is a well-known classification 
of connected algebraic groups of dimension two.
Assume that $\Bbbk$ is algebraically closed and that 
$G$ is a connected algebraic group of dimension 
two. Since no quotient group of $G$ can be 
semisimple, $G$ is solvable. Let $G_u$ be the 
unipotent radical of $G$. Then $G_u$ is normal 
and unipotent, with $G/G_u$ connected and with 
no unipotent elements.  By \cite[Section 19.1]{Hum},
$G_u$ is connected, and there is a short exact sequence
\begin{equation}
\label{E6.2.1}\tag{E6.2.1}
1\to G_u \to G\to T\to 1
\end{equation}
where $T$ is a torus. There are three cases to consider.

Case 1: $\dim G_u=0$. Then $G=T$ is a torus and
$\mathcal{O}(G)=\Bbbk[X_1^{\pm 1},X_2^{\pm 1}]$.

Case 2: $\dim G_u=2$. Then $G$ is unipotent, so $G$ is 
classified in \eqref{E6.1.1}.

Case 3: $\dim G_u=1$. By \eqref{E6.2.1}, $G$ is a semidirect 
product $(\Bbbk,+)\rtimes (\Bbbk \setminus \{0\},\times)$.
Dual to \eqref{E6.2.1},
there is a short exact sequence of commutative
Hopf algebras
$$\Bbbk \to \Bbbk[X_1^{\pm 1}]\to 
\mathcal{O}(G)\to \Bbbk[X_2]\to \Bbbk.
$$
Then $\mathcal{O}(G)$ is one of the Hopf algebras
$K(a,0,1)$ given in the last section. 

\subsection{Connected affine Hopf algebra domains of GKdimension 2}
\label{xxsec6.5} 
We do not know whether Propositions \ref{xxpro5.6} and 
\ref{xxpro5.11} give a classification of all connected 
affine Hopf $\Bbbk$-algebra domains of Gel'fand-Kirillov 
dimension 2. More precisely, one can ask: 

\begin{question}
\label{xxque5.13}
If $\Bbbk$ has positive characteristic, is every affine 
connected Hopf $\Bbbk$-algebra domain $H$ of 
Gel'fand-Kirillov dimension 2 an IHOE?
\end{question}

The answer  is ``yes'' when $H$ is commutative, and not 
only in dimension 2 - this follows from the structure 
of connected unipotent groups, Theorem \ref{xxthm6.1}. 
If one drops the restriction to dimension 2 in 
Question \ref{xxque5.13}, then the answer is 
``no'' - for example the enveloping algebra of 
$\mathfrak{sl}(3, \Bbbk)$ endowed with its standard 
cocommutative coproduct  is not an IHOE, since 
$\mathfrak{sl}(3, \Bbbk)$ does not contain a full 
flag of Lie subalgebras.

\subsection{Automorphism groups of 2-step IHOEs}
\label{xxsec6.6} 
From Proposition \ref{xxpro5.11}(3) one can easily 
read off the structure of the group $\mathrm{Aut}(H)$ 
of Hopf algebra automorphisms of 
$H := H({\bf d_s,b_s,c_{s,t}})$, for each non-trivial 
$H$. Namely, define subgroups $T$ and $N$ of 
$\mathrm{Aut}(H)$ by 
$$ T = \{\phi \in \mathrm{Aut}(H) \mid \phi (X_1) 
= \alpha X_1, \, \phi (X_2) = \beta X_2 \}, $$
where $\alpha, \beta \in (\Bbbk, \times)$ 
satisfy \eqref{E5.11.2}; and 
$$ N = \{ \phi \in \mathrm{Aut}(H) \mid \phi (X_1) 
= X_1, \, \phi(X_2) = X_2 + \sum_{s \geq 0} e_s X_1^{p^s} \},$$
for an arbitrary sequence $(e_s) \in \Bbbk^{\mathbb{N}}$ 
with only finitely many nonzero entries. Thus $N$ 
is a normal subgroup of $\mathrm{Aut}(H)$ with 
$N \cong (\Bbbk, +)^{\mathbb{N}}$. The constraints 
\eqref{E5.11.2} mean that $T \cong (\Bbbk, \times)$ 
if $H$ has only one nonzero parameter, while $T$ is 
a finite (possibly trivial) subgroup of $(\Bbbk,\times)$ 
if there are two or more defining parameters for $H$. 
It is clear from Proposition \ref{xxpro5.11} that 
$\mathrm{Aut}(H)$ is the semidirect product of $N$ by $T$.

\section{Noncommutative binomial theorem in characteristic $p$}
\label{xxsec7}

In this short section we derive a corollary of 
an important theorem due to Jacobson, which is needed in 
$\S$\ref{xxsec8}. Stated in our notation, Jacobson's theorem 
is as follows:

\begin{theorem} \cite[pp. 186-7]{Ja3}
\label{xxthm7.1}  
Let $F$ be a field of characteristic $p>0$ 
and let $R$ be a $F$-algebra. For elements $a,b$ of $R$,
$$  (a + b)^p = a^p + b^p + \sum_{i=1}^{p-1}s_i(a,b),$$
where, for $i = 1, \ldots , p-1$ and $\lambda \in F \setminus \{0\}$, 
$s_i(a,b)$ is the coefficient of 
$\frac{1}{i}\lambda^{i-1}$ in $\mathrm{ad}_{(\lambda a + b)}^{p-1}(a)$.
\end{theorem}

Here is the required corollary.

\begin{corollary}
\label{xxcor7.2} 
In the situation of the theorem, suppose that the elements 
$\mathrm{ad}_{b}^i(a)$ of $R$ commute with $a$ for 
$i = 1, \ldots , p-1$. Then
$$(a + b)^p = a^p + b^p + \mathrm{ad}_{b}^{p-1}(a). $$
\end{corollary}

\begin{proof} The extra hypothesis of the corollary ensures 
that, for $j > 0$, $\lambda^j$ does not occur in the 
expansion of $\mathrm{ad}_{(\lambda a + b)}^{p-1}(a)$. 
So the theorem states that only $s_1(a,b)$ is non-zero; namely, 
$s_1(a,b) = \mathrm{ad}_{b}^{p-1}(a)$, as claimed.
\end{proof}

\section{Properties of 2-step IHOEs}
\label{xxsec8}

\subsection{The antipode, the center, and the Calabi-Yau property}
\label{xxsec8.1} 
In this subsection, we study properties of 
the Hopf algebras listed in Proposition \ref{xxpro5.6}. 
The definition and relevant properties of the 
PI-degree are recalled in $\S$\ref{xxsec3.1}; 
regarding the Nakayama automorphism and the skew 
Calabi-Yau property, see $\S$\ref{xxsec3.2} and 
\cite{BZ1}.

\begin{proposition}
\label{xxpro8.1}
Let $H$ be a 2-step IHOE $H({\bf d_s,b_s,c_{s,t}})$ 
as in Proposition \ref{xxpro5.6}.
\begin{enumerate}
\item[(1)]
$S^2$ is the identity of $H$.
\item[(2)] 
$H$ is commutative if and only if ${\bf d_s} = {\bf 0 }$.
\item[(3)] 
$H$ is cocommutative if and only if 
${\bf c_{s,t}} = {\bf 0}$ or $p = 2$.
\item[(4)] 
Suppose that $H$ is not commutative. Then the 
center of $H$ is 
$\Bbbk[X_1^p, X_2^p-d_0^{p-1}X_2]$. Hence the 
PI-degree of $H$ is $p$. 
\item[(5)]
The Nakayama automorphism $\mu$ of $H$ is 
determined by 
$$\mu(X_1)=X_1,\quad {\text{and}} \quad \mu(X_2)=X_2+d_0.$$
Hence, $H$ is Calabi-Yau if and only if 
$d_0=0$. 
\item[(6)]
The left homological integral $\inth^{l}_{H}$ of $H$ is 
the one-dimensional $H$-bimodule which is trivial as 
left module and with
$\inth^{l}_{H} \cong H/(X_1, X_2-d_0)$ as right module.
\end{enumerate}
\end{proposition}

\begin{proof}
(1) Using the comultiplication formula of $H({\bf d_s,b_s,c_{s,t}})$
or the definition given before Theorem \ref{xxthm0.3} one finds that
$$S(X_1)=-X_1, \quad {\text{and}}\quad 
S(X_2)=-X_2+ f(X_1)$$
for some polynomial $f$. It is clear that $S^2(X_1)=X_1$.
If $p=2$, then
$$S^2(X_2)=X_2+2 f(X_1)=X_2$$
which implies that $S^2$ is the identity. (In fact, when
$p=2$, $S(X_2)=X_2+\sum_{s\geq 0} b_s X_1^{2^{s+1}}$.)
Suppose now that $p\geq 3$. Again by the 
comultiplication formula, 
$$S(X_2)=-X_2-m((S\otimes \mathrm{Id})(w))$$
where $w$ is defined in \eqref{E5.5.1} and $m$ denotes 
multiplication in $H$. Since $p$ is odd, 
$$m((S\otimes \mathrm{Id})(Z_s))=0 \qquad
\textit{ and } \qquad m((S\otimes \mathrm{Id})(Y_{s,t}))=0.$$
Therefore $S(X_2)=-X_2$ and $S^2$ is the identity.

\noindent 
(2) This is clear from \eqref{E0.2.1}.

\noindent 
(3) This is clear from \eqref{E0.2.2}.

\noindent 
(4) Since $\sigma=\mathrm{Id}$, $\delta$ is a derivation of
$\Bbbk[X_1]$. Thus 
$\delta(X_1^p)=p X_1^{p-1}\delta(X_1)=0$. This implies 
that $X_2 X_1^p=X_1^p X_2$, so that $X_1^p$ is central.

One easily checks by induction that 
\begin{equation}
\label{E8.1.1}\tag{E8.1.1}
\delta^n(X_1)=d_0^{n-1} \delta(X_1)
\end{equation}
for all $n\geq 2$. Write the derivation 
$\delta = [X_2, -]$ of $\Bbbk [X_1]$ as 
$\lambda_{X_2} - \rho_{X_2}$ where the symbols $\lambda_{X_2}$ 
[resp. $\rho_{X_2}$] denote left [resp. right] multiplication 
by $X_2$ in $H$. Note that these linear maps commute in 
$\mathrm{End}_{\Bbbk}(H)$, so one has 
$$ (\lambda_{X_2} - \rho_{X_2})^p = 
\lambda_{X_2}^p - \rho_{X_2}^p. $$
In other words, 
\begin{equation}
\label{E8.1.2}\tag{E8.1.2} 
\delta^p = [X_2^p,-],
\end{equation}
so that
$$X_2^p X_1=X_1 X_2^p+\delta^p(X_1)=
X_1X_2^p+ d_0^{p-1} \delta(X_1).$$
Combining this with the relation
$$X_2 X_1=
X_1X_2+ \delta(X_1)$$
it follows that $X_2^p-d_0^{p-1}X_2$ is central.
Thus $H$ is a free module of rank $p^2$ over the 
central polynomial subalgebra 
$Z_0 := \Bbbk[X_1^p,X_2^p-d_0^{p-1} X_2]$. Hence, 
using $Q(-)$ to denote quotient division algebras,
$$ \mathrm{dim}_{Q(Z_0)} Q(H) = p^2.$$
But $\mathrm{dim}_{Q(Z(H))} Q(H)$ is an even power 
of $p$, by Theorem \ref{xxthm3.3} and the discussion 
at the start of $\S$\ref{xxsec3.1}. Therefore, 
since $H$ is not commutative, the PI-degree of $H$ is 
$p$ and $Q(Z(H)) = Q(Z_0)$. 
However, $Z(H)$ is a finite module over $Z_0$ by 
noetherianity of $H$ as a $Z_0$-module, and $Z(H)$ 
is contained in $Q(Z_0)$. Since $Z_0$ is normal it 
follows that $Z(H) = Z_0$ . That the PI-degree of $H$ is $p$ is now clear.

\noindent 
(5) By \cite[Theorem 4.2]{LiM},
$\mu(X_1)=X_1$ and 
$$\mu(X_2)=X_2+{\frac{d}{d X_1}} (\delta(X_1))=X_2+d_0.$$ 
The consequence is clear
(and it also follows from \cite[Corollary 4.3]{LiM}).

\noindent 
(6) By \cite[Theorem 0.3]{BZ1}, $\mu=S^2 \circ \Xi^{l}_{\eta}$
where $\eta$ is the right character of the left 
homological integral $\inth^l$ of $H$ and $\Xi^l_{\eta}$
is the corresponding left winding automorphism. By 
(1), $S^2=\mathrm{Id}$, so that $\mu=\Xi^{l}_{\eta}$. 
Thus (5) implies (6).
\end{proof}

\subsection{Representation theory of 2-step IHOEs}
\label{xxsec8.2} 
In this subsection we describe the simple 
representations of the 2-step IHOEs 
$H:= H({\bf d_s,b_s,c_{s,t}})$. Recall that if 
$\Bbbk$ is any algebraically closed field and $A$ 
is an affine $\Bbbk$-algebra which is a finite module 
over its center $Z$, then $Z$ is also affine by the 
Artin-Tate lemma \cite[Lemma 13.9.10]{McR}, and a 
version of Schur's lemma applies \cite[Theorem 13.10.3]{McR}: 
namely, if $V$ is a simple $A$-module, then 
$\mathrm{End}_A(V) = \Bbbk$,  and so $V$ is annihilated 
by a maximal ideal $\mathfrak{m}$ of $Z$, so that $V$ is 
a (necessarily finite dimensional) simple module over 
the finite dimensional algebra $H/\mathfrak{m}H$. If $A$ 
is prime (as in the current setting, when 
$A= H({\bf d_s,b_s,c_{s,t}})$ is a domain), with 
$\mathrm{PIdeg}(A) = d$, then for $\mathfrak{m}$ 
in a non-empty (and hence dense) open subset 
$\mathcal{A}(A)$ of $\mathrm{maxspec}(A)$, 
$$ A/\mathfrak{m}A \cong M_d (\Bbbk). $$
The set $\mathcal{A}(A)$ is called the 
\emph{Azumaya locus} of $A$. For convenience, we shall 
denote the \emph{non-Azumaya locus} by $\mathcal{NA}(A)$; 
that is, $\mathcal{NA}(A) := \mathrm{maxspec}(Z) \setminus 
\mathcal{A}(A)$, a proper closed subset of 
$\mathrm{maxspec}(Z)$. If the simple $A$-module $V$ has 
$\mathrm{Ann}_Z(V) \in \mathcal{NA}(A)$ then 
$\mathrm{dim}_{\Bbbk}(V) < d$. For more details on this circle 
of ideas, see for example \cite[Part III]{BG}.

Suppose now that $H:= H({\bf d_s,b_s,c_{s,t}}) = 
\Bbbk[X_1][X_2 ; \delta]$, so that 
$$Z := Z(H) = \Bbbk [X_1^p, X_2^p - d_0^{p-1}X_2]$$
by Theorem \ref{xxpro8.1}(4). In view of the discussion 
in the previous paragraph, our task is to describe the 
algebras
$$H_{\alpha,\beta} := H/\mathfrak{m}_{\alpha,\beta}H,$$
where $\mathfrak{m}_{\alpha,\beta}$ denotes the maximal 
ideal 
$\langle X_1^p -\alpha, X_2^p - d_0^{p-1}X_2 - \beta\rangle$ 
of $Z$ and $(\alpha,\beta)$ ranges through 
$\mathbb{A}^2 (\Bbbk)$. Notice that, by the PBW theorem for 
$H$, 
$$\mathrm{dim}_{\Bbbk}(H_{\alpha,\beta}) = p^2$$
for all $(\alpha,\beta) \in \mathbb{A}^2(\Bbbk)$. Therefore 
the maximal dimension of a simple $H$-module $V$ is $p$, 
and this value is attained by $V$ if and only if $V$ is 
the (unique) simple $H/\mathfrak{m}_{\alpha,\beta}H$-module 
for some $(\alpha,\beta) \in \mathcal{A}(H)$. 

Recall from \eqref{E0.2.1} that $\delta (X_1) = 
d_0 X_1 + \sum_{s \geq 1}d_sX_1^{p^s}$. 
It is convenient to define a polynomial 
$d(x) = \sum_{s \geq 1} d_s x^{p^{s-1}}$, so that
\begin{equation}
\notag%\label{E8.1.3}\tag{E8.1.3} 
\delta (X_1)^p 
= d_0^p X_1^p + (\sum_{s \geq 1}d_s X_1^{p^s})^p 
= d_0^p X_1^p + d(X_1^p)^p.
\end{equation}

\begin{proposition}
\label{xxpro8.2} 
Suppose $\Bbbk$ is algebraically closed of 
characteristic $p>0$. Let $H$ be a 2-step IHOE as 
given in Proposition \ref{xxpro5.6}, and fix the 
notation as above. Suppose that $H$ is not 
commutative, that is ${\bf d_s} \neq {\bf 0}$.
\begin{enumerate}
\item[(1)] 
The defining ideal of $\mathcal{NA}(H)$ in $Z$ is 
$\sqrt{\langle d_0^p X_1^p + d(X_1^p)^p \rangle}$. 
That is, 
$$H_{\alpha,\beta} \cong M_p(\Bbbk) \; 
\Longleftrightarrow \; d_0^p \alpha + d(\alpha)^p \neq 0. $$
\item[(2)] 
Suppose that $d_0 = 0$. Let $\mathfrak{m}_{\alpha,\beta} 
\in \mathcal{NA}(H)$; that is, 
$d_0^p \alpha + d(\alpha)^p = 0,$ or equivalently,
$$ d(\alpha) = 0.$$ Then 
$$ H_{\alpha,\beta} \cong \Bbbk [X,Y] / 
\langle X^p, Y^p \rangle.$$
Thus $H_{\alpha,\beta}$ has a unique simple module, 
of dimension 1.
\item[(3)] 
Suppose that $d_0 \neq 0$. Let 
$\mathfrak{m}_{\alpha,\beta} \in \mathcal{NA}(H)$, 
so $d_0^p \alpha + d(\alpha)^p = 0.$ Then there 
are elements $u, w \in H_{\alpha,\beta}$ such that 
$H_{\alpha,\beta} = \Bbbk \langle u,w \rangle$, 
with relations
\begin{equation}
\label{E8.2.1}\tag{E8.2.1}
u^p = 0, \quad 
wu - uw = u, \quad 
w^p - w + \alpha \beta d(\alpha)^{-p} = 0.
\end{equation}
In particular, $H_{\alpha,\beta}$ has a single 
block of $p$ simple modules, each of dimension 1.
\item[(4)] 
The non-Azumaya locus $\mathcal{NA}(H)$ is the 
disjoint union of $r$ copies of $\mathbb{A}^1(\Bbbk)$, 
the affine line, where $r$ is the number of 
distinct roots of the equation 
$$\sum_{s \geq 0} d_s^p x^{p^s} = 0.$$ 
\end{enumerate}
\end{proposition}

\noindent Note that in the setting of part (3) the 
third equation in \eqref{E8.2.1} is equivalent to
$$w^p-w=\beta d_0^{-p}.$$

\begin{proof} 
(1,2,3) 
Let $(\alpha,\beta) \in \mathbb{A}^2(\Bbbk)$. 
The defining relation of 
$H := H({\bf d_s, b_s, c_{s,t}})$, namely 
$[X_2,X_1] = \delta (X_1)$, yields in 
$H_{\alpha,\beta}$ the relation 
\begin{equation}
\label{E8.2.2}\tag{E8.2.2}
X_2 X_1-X_1X_2=\delta(X_1)=d_0 X_1 
+\sum_{s\geq 1} d_s \alpha^{p^{s-1}}
=d_0 X_1+d(\alpha).
\end{equation}
Suppose first that 
$d_0^p \alpha + d(\alpha)^p \neq 0$. If 
$d_0 = 0$, then \eqref{E8.2.2} yields the 
defining relation for the Weyl algebra. 
Thus $H_{\alpha,\beta}$ is a factor of 
the first Weyl algebra over $\Bbbk$, of 
dimension $p^2$. But the Weyl $\Bbbk$-algebra 
is Azumaya of PI-degree $p$ by 
\cite[Theorem 2]{Re}, so each of its 
$p^2$-dimensional factors, in particular 
$H_{\alpha,\beta}$, is isomorphic to $M_p(\Bbbk)$.

Next consider the case where 
$d_0^p \alpha + d(\alpha)^p \neq 0$ and 
$d_0 \neq 0$. Let $u$ and $w$ denote the 
images in $H_{\alpha,\beta}$ of 
$d_0 X_1+d(\alpha)$ and
$d_0^{-1}X_2$ respectively. These satisfy 
the following relations in $H_{\alpha,\beta}$:
$$\begin{aligned}
w u - u w&= u,\\
u^p&=d_0^p \alpha+d(\alpha)^p,\\
w^p- w&= d_0^{-p} \beta.
\end{aligned}
$$
The second of these relations shows that 
$u$ is a unit, so post-multiplying the 
first relation by $u^{-1}$ again yields 
the defining relation of the first Weyl 
algebra. As before we find that 
$H_{\alpha,\beta} \cong M_p(\Bbbk)$. We 
have therefore shown that $\mathcal{NA}(H)$ 
is contained in the subvariety of 
$\mathrm{maxspec}(Z)$ defined by 
$d_0^pX_1^p + d(X_1^p)^p$. 

Now assume that $d_0^p \alpha + d(\alpha)^p = 0$. 
Suppose first that $d_0 = 0$. Then 
$Z = \Bbbk [X_1^p,X_2^p]$, and so the images in 
$H_{\alpha,\beta}$ of $X_1 - \alpha^{\frac{1}{p}}$ 
and $X_2 - \beta^{\frac{1}{p}}$ are mutually 
commuting nilpotent elements , which together 
generate $H_{\alpha,\beta}$ and have index of 
nilpotency $p$. Therefore $H_{\alpha,\beta}$ 
is the local algebra 
$\Bbbk [X,Y]/\langle X^p, Y^p \rangle$ in this case. 

Finally, suppose that $d_0^p \alpha + d(\alpha)^p = 0$, 
with $d_0 \neq 0$. Let $u$ and $w$ denote the 
images in $H_{\alpha,\beta}$ of 
$d_0 X_1+d(\alpha)$ and
$d_0^{-1}X_2$ respectively.
Then we obtain the three relations in \eqref{E8.2.1}.
The first two relations of \eqref{E8.2.1} now show that 
$u$ is a nilpotent normal element, with index of 
nilpotency at most $p$. Now $H_{\alpha,\beta}/uH_{\alpha,\beta} 
= \Bbbk \langle w \rangle$, and the third relation 
of \eqref{E8.2.1} has the form $f(w) = 0$ where $f(x)$ is a 
polynomial with $f'(x) = -1$. Hence $f(w) = 0$ 
has no repeated roots (or $f(w)=0$ has root $w, w+1,\cdots,w+p-1$
when one root $w$ is chosen), so that 
$$ H_{\alpha,\beta}/u H_{\alpha,\beta} \cong \Bbbk^{\oplus p}.$$
Since $\mathrm{dim}_{\Bbbk}H_{\alpha,\beta} = p^2$, 
this shows that $u^p = 0 \neq u^{p-1}$. Moreover 
the second relation of \eqref{E8.2.1} shows that the 
$p$ 1-dimensional simple $H_{\alpha,\beta}$-modules form 
a single block, the  $\mathrm{Ext}$-quiver being a circle. 
This completes the proof of parts (1), (2) and (3).

\noindent 
(4) From the above proof we see that a maximal 
ideal $\mathfrak{m}_{\alpha,\beta}$ of $Z$ is 
in $\mathcal{NA}(H)$ if and only if $\alpha$ is 
a root of the equation $d_0^p x + d(x)^p = 0$, 
that is of the equation 
$d_0^p x + \sum_{s\geq 1}d_s^p x^{p^s} = 0$. 
This proves (4).
\end{proof}

\subsection{The Hopf center and restricted Hopf algebras}
\label{xxsec8.3} 
Recall the following definition, due to Andruskiewitsch, 
\cite[Definition 2.2.3]{An}.

\begin{definition}
\label{xxdef8.3} 
Let $H$ be a Hopf algebra. The \emph{Hopf center} of $H$, 
denoted $C(H)$, is the unique largest central Hopf 
subalgebra of $H$.
\end{definition}

\noindent It is easy to see that $C(H)$ exists for any Hopf algebra 
$H$. As already discussed in Remark \ref{xxrem2.5}(1), 
when $\Bbbk$ has positive characteristic $p$ and 
$\mathfrak{g}$ is an $n$-dimensional Lie algebra over 
$\Bbbk$, $C(U(\mathfrak{g}))$ exists and is a polynomial 
algebra in $n$ variables, with $U(\mathfrak{g})$ a free 
$C(U(\mathfrak{g}))$-module of rank a power of $p$. In 
general, even when $H$ is a finite module over its center, 
$C(H)$ can be very small - consider, for instance, the 
group algebra $H = FG$ over any field $F$ of the dihedral 
group 
$$G = \langle a,b : b^2 = 1, bab = a^{-1} \rangle,$$ 
where $C(H) = F$. Nevertheless, current evidence 
suggests that when $H$ is a connected Hopf $\Bbbk$-algebra 
$C(H)$ may always be large, and Question \ref{xxque2.7} 
proposes that this may be the case for IHOEs over 
$\Bbbk$. We shall show in this subsection that this is 
indeed the case for all 2-step $\Bbbk$-IHOEs. First, 
however, we show that whenever a $\Bbbk$-IHOE \emph{is} 
a finite module over its Hopf center, some of the 
desirable features of the Lie algebra case immediately 
follow. In this subsection we denote the augmentation 
ideal of a Hopf algebra $T$ by $T_{+}$. 

\begin{proposition}
\label{xxpro8.4} 
Let $\Bbbk$ be algebraically closed of characteristic 
$p > 0$, let $n$ be a positive integer, and let $H$ 
be a noncommutative $n$-step IHOE over $\Bbbk$. 
Suppose that $H$ is a finite $C(H)$-module. Then 
$C(H)$ is a polynomial algebra in $n$ variables, 
and $H$ is a free $C(H)$-module of rank $p^{\ell}$ 
for some $\ell \geq 2$.
\end{proposition}

\begin{proof} By \cite[Proposition 2.5]{BOZZ}, $H$ is 
connected as a Hopf algebra, so $C(H)$ is also connected. 
Moreover, $C(H)$ is affine of Gel'fand-Kirillov dimension 
$n$, by the Artin-Tate lemma and Corollary \ref{xxcor2.4}. 
By Theorem \ref{xxthm6.1} ((3) $\Rightarrow$ (2)),
$C(H)$ is a commutative polynomial ring $\Bbbk[X_1,\cdots,X_n]$. 
By \cite[Theorem 0.3]{WZ}, $H$ is a finitely generated 
projective module over $C(H)$. Hence
$H$ is a free module over $C(H)$ of finite rank.
Let $r$ denote the rank of $H$ over $C(H)$. 
By definition, $C(H)$ is central in $H$, then
$\overline{H}: = H/(C(H))_{+} H$ is a Hopf
algebra of dimension $r$. By \cite[Corollary 5.3.5]{Mo1},
$\overline{H}$ is connected. Since $\overline{H}$ 
is finite dimensional and connected, its dimension
is of the form $p^{\ell}$ for some 
$\ell>0$ \cite[Proposition 2.2(7)]{Wa1}. 

Since $H$ is prime, the rank $H$ as a $Z(H)$-module 
is a square of some integer. This implies that $r$
is not a prime number. Therefore $\ell\geq 2$.
\end{proof}

Observe that, in the setting of the proposition, the 
factor algebra 
\begin{equation}\notag%\label{E8.4.1}\tag{E8.4.1}
\overline{H} := H/C(H)_{+} H
\end{equation} 
is a connected Hopf algebra of dimension $p^{\ell}$. 
It seems reasonable to call $\overline{H}$ the 
\emph{restricted Hopf algebra of} $H$. 

In the rest of this subsection we confirm that the 
hypotheses of the proposition are satisfied by all 
noncommutative 2-step IHOEs, with $\ell$ equal to 2 
or 3 in all cases, and we determine their restricted 
Hopf algebras. The following notation will remain in 
force throughout the rest of $\S$\ref{xxsec8.3}.

\begin{notation}
\label{xxnot8.5} 
The field $\Bbbk$ is algebraically closed of 
characteristic $p > 0$, and $H$ denotes a 2-step 
IHOE $H({\bf d_s,b_s,c_{s,t}}) = 
\Bbbk \langle X_1, X_2 \rangle$ as defined in 
Proposition \ref{xxpro5.6}, see also 
$\S$\ref{xxsec0.3}. The element 
$X_2^p -d_0^{p-1} X_2$ of $H$, which is central 
by Proposition \ref{xxpro8.1}(4), will be denoted by $z$.  

\noindent 
We will always denote the element of 
$H \otimes H$ defined in \eqref{E0.2.2} by $w$, 
so that, writing $\lambda_i$ for 
$\frac{(p-1)!}{i!(p-i)!}$, $(1 \leq i \leq p-1)$,
$$ w = \sum_{s \geq 0} b_s \left(\sum_{i=1}^{p-1} 
\lambda_i (X_1^{p^s})^i \otimes (X_1^{p^s})^{p-i}\right) 
+ \sum_{0 \leq s < t}c_{s,t}\left(X_1^{p^s} 
\otimes X_1^{p^t} - X_1^{p^t} \otimes X_1^{p^s}\right).$$
We denote the element  $X_2\otimes 1+1\otimes X_2$ 
of $H \otimes H$ by $b$.

\noindent 
Let $s$ be a positive integer divisible by $p$, 
and let $\equiv_s$ denote equality of elements of 
$H \otimes H$ modulo the subspace 
$ X_1^{s} H\otimes H+H\otimes X_1^{p} H$. Since this 
subspace is preserved by $ad_b$ it follows that, for 
elements $f$ and $g$ of $H \otimes H$,
\begin{equation}
\label{E8.5.1}\tag{E8.5.1} 
f\equiv_s g \quad \Longrightarrow  
\quad ad_b(f)\equiv_s ad_b(g).
\end{equation}
\end{notation}

Using the fact
$$ad_{X_2}( X_1^i)=
i X_1^{i-1} \delta(X_1)$$
for all $i\geq 1$, one can check that
\begin{align}
ad_b(w)&\notag\\
\notag
&=[\delta(X_1)\otimes 1]
\left(b_0 \{(X_1\otimes 1+1\otimes X_1)^{p-1}-X_1^{p-1}\otimes 1\}
+\sum_{t>0} c_{0,t} 1\otimes X_1^{p^t}\right)\\
\label{E8.5.2}\tag{E8.5.2}&
+[1\otimes \delta(X_1)]
\left(b_0\{(X_1\otimes 1+1\otimes X_1)^{p-1}-1\otimes X_1^{p-1}\}
-\sum_{t>0} c_{0,t} X_1^{p^t}\otimes 1\right).
\end{align}

\begin{lemma}
\label{xxlem8.6} 
Retain Notation \ref{xxnot8.5}.
\begin{enumerate}
\item[(1)] 
$\Delta(z)= 
z\otimes 1+1\otimes z+w^p- d_0^{p-1} w+ad_b^{p-1}(w).$ 
Moreover, the element $w^p- d_0^{p-1} w+ad_b^{p-1}(w)$ 
of $H \otimes H$ is in 
$\Bbbk[X_1]\otimes \Bbbk[X_1]$.
\item[(2)]
The subalgebra $\Bbbk[X_1^p, z]$ of $H$ is a Hopf 
subalgebra if and only if
$$- d_0^{p-1} w+ad_b^{p-1}(w)\in 
\Bbbk[X_1^p]\otimes \Bbbk[X_1^p].$$
\item[(3)]
If $b_0=0$, then $\Bbbk[X_1^p, z]$ is a Hopf 
subalgebra of $H$.
\end{enumerate}
\end{lemma}

\begin{proof} (1) First of all, for every $i\geq 0$,
$ad^i_b(w)\in \Bbbk[X_1]\otimes \Bbbk[X_1]$. Hence
$\{ad^i_b(w)\}_{i\geq 0}$ are in a commutative subalgebra
of $H\otimes H$. Thus all the hypotheses of 
Corollary  \ref{xxcor7.2} are satisfied. Therefore, by that corollary 
and the definitions of $H$ and $z$,  
$$\begin{aligned}
\Delta(z)&=\Delta(X_2^p)-d_0^{p-1} \Delta(X_2)
=(\Delta(X_2))^p-d_0^{p-1} \Delta(X_2)\\
&=(w+b)^p-d_0^{p-1} (w+b)\\
&=w^p+ b^p+ad_b^{p-1}(w) -d_0^{p-1}(w+b)\\
&=w^p + (X_2\otimes 1+1\otimes X_2)^p 
+ad_b^{p-1}(w)-d_0^{p-1}(X_2\otimes 1+1\otimes X_2+w)\\
&=w^p + (X_2\otimes 1)^p+(1\otimes X_2)^p 
+ad_b^{p-1}(w)-d_0^{p-1}(X_2\otimes 1+1\otimes X_2+w)\\
&=z\otimes 1+ 1\otimes z+ w^p -d_0^{p-1} w+ad_b^{p-1}(w).
\end{aligned}
$$
That $w^p- d_0^{p-1} w+ad_b^{p-1}(w)$ is in 
$\Bbbk[X_1]\otimes \Bbbk[X_1]$ is clear from the definition of $w$.

(2) Denote $\Bbbk[X_1^p,z]$ by $Z$.
Since a connected bialgebra is a Hopf algebra by 
\cite[Corollary 3.5.4(a)]{HS}, it is enough to 
check when $\Delta (Z) \subseteq Z \otimes Z$. 
As $X_1^p$ is primitive, this amounts to checking 
whether $\Delta(z) \in Z \otimes Z$. The result now follows 
from part (1) and the fact that $w^p\in 
\Bbbk[X_1^p]\otimes \Bbbk[X_1^p]$.

(3) If $b_0=0$, then, by 
\eqref{E8.5.2}, 
$$ad_b(w)=[\delta(X_1)\otimes 1]
\left(\sum_{t>0} c_{0,t} 1\otimes X_1^{p^t}\right)
+[1\otimes \delta(X_1)]
\left(-\sum_{t>0} c_{0,t} X_1^{p^t}\otimes 1\right).$$
By \eqref{E8.1.1} and the fact that $ad_{X_2}(X_1^{pi})=0$ 
for all $i \geq 0$, we have that
$$ad_b^n(w)=[d_0^{n-1}\delta(X_1)\otimes 1]
\left(\sum_{t>0} c_{0,t} 1\otimes X_1^{p^t}\right)
+[1\otimes d_0^{n-1}\delta(X_1)]
\left(-\sum_{t>0} c_{0,t} X_1^{p^t}\otimes 1\right)$$
for all $n\geq 1$. Then
$$\begin{aligned}
ad^{p-1}_b(w)
&=d_0^{p-2} \left\{(\delta(X_1)\otimes 1)
\left(\sum_{t>0} c_{0,t} 1\otimes X_1^{p^t}\right)
-(1\otimes \delta(X_1)) \left(\sum_{t>0} c_{0,t} 
X_1^{p^t}\otimes 1\right)\right\}
\\
&= d_0^{p-1}\left\{\sum_{t>0} c_{0,t} (X_1\otimes X_1^{p^t}
-X_1^{p^t}\otimes X_1)\right\} +{\mathcal T}_1
\end{aligned}
$$
where ${\mathcal T}_1\in \Bbbk[X_1^p]\otimes \Bbbk[X_1^p]$.
On the other hand, by definition, 
$$\begin{aligned}
d_0^{p-1} w &= 
d_0^{p-1} \left\{\sum_{t>0} c_{0,t} (X_1\otimes X_1^{p^t}
-X_1^{p^t}\otimes X_1)\right\} +{\mathcal T}_2
\end{aligned}
$$
where ${\mathcal T}_2\in \Bbbk[X_1^p]\otimes \Bbbk[X_1^p]$.
Hence
$$-d_0^{p-1} w+ad_{b}^{p-1}(w) 
\in \Bbbk[X_1^p]\otimes \Bbbk[X_1^p].$$
The assertion follows from part (2).
\end{proof}

\begin{lemma}
\label{xxlem8.7} 
Retain Notation \ref{xxnot8.5}.
\begin{enumerate}
\item[(1)] 
For all $i > 0$,
\begin{equation}
\label{E8.7.1}\tag{E8.7.1} 
ad_b^i(w)\equiv_{p} 0 .
\end{equation}
\item[(2)]
Suppose $b_0\neq 0$ and $d_0 \neq 0$. Then 
\begin{equation}
\notag%\label{E8.7.2}\tag{E8.7.2} 
-d_0^{p-1} w+ad_b^{p-1}(w)\not\in 
\Bbbk[X_1^p]\otimes \Bbbk[X_1^p].
\end{equation}
Hence $\Bbbk[X_1^p, z]$ is not a Hopf subalgebra of $H$.
\item[(3)]
Suppose that $d_0 =0$, that $b_0\neq 0$ and 
that $\delta(X_1)\neq 0$. 
Then $\Bbbk[X_1^p, z]$ is not a Hopf subalgebra of $H$.
\end{enumerate}
\end{lemma}

\begin{proof} (1) Using \eqref{E8.5.2}, 
$$\begin{aligned} 
ad_b(w) 
&\equiv_p \; (\delta(X_1) \otimes 1)b_0 
\{(X_1 \otimes 1 + 1 \otimes X_1)^{p-1} - X_1^{p-1} \otimes 1 \}\\
& \quad \quad + (1 \otimes \delta(X_1))b_0 
\{(X_1 \otimes 1 + 1 \otimes X_1)^{p-1} - 1 \otimes X_1^{p-1} \}\\
&\equiv_p d_0 b_0(X_1 \otimes 1)\{(X_1 \otimes 1 
+ 1 \otimes X_1)^{p-1} - X_1^{p-1} \otimes 1 \}\\
& \quad \quad + d_0 b_0(1 \otimes X_1)\{(X_1 \otimes 1 
+ 1 \otimes X_1)^{p-1} - 1 \otimes X_1^{p-1} \}\\
&\equiv_p \; d_0 b_0 (X_1 \otimes 1 + 1 \otimes X_1)^p 
- d_0 b_0(X_1^p \otimes 1) - d_0 b_0 (1 \otimes X_1^p)\\
&\equiv_p \; 0.
\end{aligned}$$
This proves  \eqref{E8.7.1} for $i = 1$, and the general 
case follows from \eqref{E8.5.1}.

(2) Assume that $d_0$ and $b_0$ are non-zero. Then, using 
\eqref{E8.7.1},
$$\begin{aligned} 
-d_0^{p-1} w+ad_b^{p-1}(w)  &\equiv_p \; -d_0^{p-1}w\\
&\equiv_p \; -d_0^{p-1} b_0\left(\sum_{i=1}^{p-1}
\lambda_iX_1^i \otimes X_1^{p-i}\right).
\end{aligned}
$$
This implies that 
\begin{equation}
\notag
-d_0^{p-1} w+ad_b^{p-1}(w) \not\in \Bbbk [X_1^p] \otimes \Bbbk [X_1^p].
\end{equation}
The second claim now follows from Lemma \ref{xxlem8.6}(2).

(3) Using \eqref{E8.5.2} and the hypothesis $d_0=0$ we obtain
$$\begin{aligned} ad_b (w) 
&=[\delta(X_1)\otimes 1]
\left(b_0 \{(X_1\otimes 1+1\otimes X_1)^{p-1}-X_1^{p-1}\otimes 1\}
\right)\\
&\quad  
+[1\otimes \delta(X_1)]
\left(b_0\{(X_1\otimes 1+1\otimes X_1)^{p-1}-1\otimes X_1^{p-1}\}
\right) +{\mathcal T}
\end{aligned}
$$
where ${\mathcal T}\in \Bbbk [X_1^p] \otimes \Bbbk [X_1^p]$.
Since $d_0=0$, $\delta(X_1)\in \Bbbk [X_1^p]$ and hence is 
central. Therefore, by induction, we can show that
$$\begin{aligned}
ad^{p-1}_b(w)&= 
b_0 (p-1)! [\delta(X_1)\otimes 1]
\{(\delta(X_1)\otimes 1+1\otimes \delta(X_1))^{p-2}
(X_1\otimes 1+1\otimes X_1) \\
&\qquad - (\delta(X_1)\otimes 1)^{p-2} (X_1\otimes 1) 
\}\\
&\quad  
+b_0 (p-1)!  [1\otimes \delta(X_1)]
\{(\delta(X_1)\otimes 1+1\otimes \delta(X_1))^{p-2}
(X_1\otimes 1+1\otimes X_1) \\
&\qquad 
- (1\otimes \delta(X_1))^{p-2} (1\otimes X_1)\} 
+ \delta_{2, \mathrm{char}\Bbbk}\mathcal{T}\\
&= b_0 (p-1)! [(\delta(X_1)\otimes 1+1\otimes \delta(X_1))^{p-1}
(X_1\otimes 1+1\otimes X_1)\\
&\quad -(\delta(X_1)\otimes 1)^{p-1} (X_1\otimes 1)
-(1\otimes \delta(X_1))^{p-1} (1\otimes X_1)] 
+ \delta_{2, \mathrm{char}\Bbbk}\mathcal{T}.
\end{aligned}
$$
Observe that this element is not a member of 
$\Bbbk [X_1^p] \otimes \Bbbk [X_1^p]$. Since
$d_0=0$, $d_0^{p-1} w=0$. Therefore
$-d_0^{p-1} w+ad_b^{p-1}(w)\not\in \Bbbk [X_1^p]
\otimes \Bbbk [X_1^p]$. The assertion thus follows 
from Lemma \ref{xxlem8.6}(2).
\end{proof}

\noindent Lemmas \ref{xxlem8.6} and \ref{xxlem8.7} can 
now be applied to determine $C(H)$ for all 2-step IHOEs:

\begin{theorem}
\label{xxthm8.8} 
Retain Notation \ref{xxnot8.5}. Assume that $H$ is not
commutative.
\begin{enumerate}
\item[(1)] 
The subalgebra $\Bbbk [X_1^p, z^p]$ is a central Hopf 
subalgebra of $H$.
\item[(2)]
The Hopf center $C(H)$ of $H$ is either 
$\Bbbk [X_1^p, z^p]$ or $\Bbbk [X_1^p, z]$. 
\item[(3)] 
$C(H) = \Bbbk [X_1^p, z]$ if and only if $b_0= 0$.
\end{enumerate}
\end{theorem}

\begin{proof} 
(1) Since $H$ is connected, it is enough to show that 
$\Bbbk [X_1^p, z^p]$ is a bialgebra. First, $X_1^p$ 
is primitive. To see that $\Delta (z^p) \in 
\Bbbk [X_1^p, z^p] \otimes\Bbbk [X_1^p, z^p]$, note that 
$z$ is central and apply Lemma \ref{xxlem8.6}(1) to 
determine $\Delta(z)^p$.

\noindent 
(2) In view of part (1) and Proposition \ref{xxpro8.1}(4), 
$$\Bbbk [X_1^p, z^p] \subseteq C(H) \subseteq \Bbbk [X_1^p, z]. $$
Moreover, $H$ is a free module over each of these 
subalgebras, with the ranks over the outer two algebras 
being $p^3$ and $p^2$. But the rank of $H$ over $C(H)$ 
is also a power of $p$ by Proposition \ref{xxpro8.4}. 
Thus equality must hold at some point in the chain of inclusions.

\noindent (3) 
The assertion follows from Lemmas \ref{xxlem8.6}(3) and
\ref{xxlem8.7}(2,3).
\end{proof}

\noindent The following two propositions examine the 
restricted Hopf algebra of $H$ in each of the two 
cases distinguished by Theorem \ref{xxthm8.8}.

\begin{proposition}
\label{xxpro8.9} 
Retain Notation \ref{xxnot8.5}. Assume that $H$ is not
commutative. Suppose that $b_0=0$. In parts {\rm{(2,3)}}
let $\overline{H}$ denote the quotient Hopf algebra $H/Z_{+}H$.
\begin{enumerate}
\item[(1)] 
The center $Z:=\Bbbk [X_1^p, z]$ is a Hopf subalgebra of $H$. 
\item[(2)] 
Suppose $d_0 = 0$. Then $\overline{H}$ is
$$H_{0,0} \cong \Bbbk [x,y]/\langle x^p, y^p \rangle, $$
a commutative cocommutative Hopf algebra of 
dimension $p^2$, with $x$ and $y$ primitive.
\item[(3)] 
Suppose that $d_0 \neq 0$. Then $\overline{H}$ is
$$H_{0,0} \cong 
\Bbbk \langle x,y \, : \, [y,x] - x, 
\, x^p, \, y^p - y \rangle , $$
a cocommutative Hopf algebra of dimension $p^2$ 
with $x$ and $y$ primitive, the restricted 
enveloping algebra of the 2-dimensional 
non-abelian restricted Lie algebra.
\end{enumerate}
\end{proposition}

\begin{proof} 
The center $Z$ is as stated, by Proposition 
\ref{xxpro8.1}(4). 

\noindent 
(1) This is Theorem \ref{xxthm8.8}(3).

\noindent 
(2) By Proposition \ref{xxpro8.2}(2),
$$H/Z_{+} H \; = \; H_{0,0} \; \cong \; 
\Bbbk [x,y]/\langle x^p, y^p \rangle,$$
where $x$ and $y$ are the images of $X_1$ and $X_2$ 
respectively. Since $X_1$ is primitive so is $x$, 
and $y$ is also primitive under the hypotheses. 

\noindent 
(3) Assume that $d_0 \neq 0$. Then Proposition 
\ref{xxpro8.2}(3) applies, yielding
$$
H/Z_{+} H \; = \; H_{0,0} \; \cong \; 
\Bbbk \langle x,y : x^p, \, y^p - y, \, [y,x] = x \rangle , 
$$
where $x$ and $y$ are respectively the images of $X_1$ 
and $d_0^{-1} X_2$. As in (2), $x$ and $y$ are both primitive 
since $b_0 = 0$.
\end{proof}

\begin{proposition}
\label{xxpro8.10} 
Retain Notation \ref{xxnot8.5}. Assume that $H$ is not
commutative. Suppose that $b_0\neq 0$. In parts {\rm{(2,3)}}
let $\overline{H}$ denote the quotient Hopf algebra $H/C_{+}H$.
\begin{enumerate}
\item[(1)] 
The center $\Bbbk [X_1^p, z]$ is not a Hopf subalgebra 
of $H$ and the unique largest central Hopf subalgebra 
of $H$ is $C := \Bbbk [X_1^p, z^p]$.
\item[(2)] 
Suppose that $d_0 = 0$. Then 
$$\overline{H}
\cong \Bbbk[x,y]/(x^p, y^{p^2})$$ with $x$ primitive and 
$$ \Delta (y) = y \otimes 1 + 1 \otimes y + 
\sum_{i=1}^{p-1}\lambda_i x^i \otimes x^{p-i}.$$
This is a commutative and cocommutative Hopf algebra
of dimension $p^3$.
\item[(3)] 
Suppose that $d_0 \neq 0$. Then 
the quotient Hopf algebra $\overline{H}$ is 
$\Bbbk \langle x,y,z \rangle$, with relations
$$
\begin{aligned}     
\; [y,x] &= x,\\
z&= y^p-y,\\
[y,z] &= [x,z] = 0,\\
x^p &= z^p = 0.
\end{aligned}
$$ 
This is noncommutative and cocommutative of 
dimension $p^3$, with $x$ primitive and 
$$ \Delta (y) = y \otimes 1 + 1 \otimes y + 
b_0 d_0^{-1}\sum_{i=1}^{p-1}\lambda_i x^i \otimes x^{p-i}$$
and
$$ \Delta (z) = z \otimes 1 + 1 \otimes z -b_0d_0^{-1} 
\sum_{i=1}^{p-1}\lambda_i x^i \otimes x^{p-i}$$
\end{enumerate}
\end{proposition}

\begin{proof} 
(1) This follows from Theorem \ref{xxthm8.8}(2,3).

\noindent 
(2)  Assume that $d_0 = 0$. Let $x$ be the image of 
$X_1$ and $y$ the image of $b_0^{-1} X_2$. The result
follows by direct computation.

\noindent
(3) Assume that $d_0 \neq 0$. Let $x$ be the image of 
$X_1$, $y$ the image of $b_0^{-1} X_2$, and $z$ the image
of $(d_0^{-1}X_2)^p-(d_0^{-1}X_2)$. Most of the assertions
follow by routine computations. For example, it is not 
hard to check that
$$ \Delta (y) = y \otimes 1 + 1 \otimes y + 
e u$$
where $u=\sum_{i=1}^{p-1}\lambda_i x^i \otimes x^{p-i}$
and $e:=b_0d_0^{-1}$. Here we only prove the formula for 
$\Delta(z)$. Repeating the computation in 
Lemma \ref{xxlem8.6}(1), we have, for 
$ad_b:=ad_{y\otimes 1+1\otimes y}$,
$$\begin{aligned}
\Delta(z)&=z\otimes 1+ 1\otimes z+(e u)^p -e u +ad_{b}^{p-1}(eu)\\
&= z\otimes 1+ 1\otimes z -e u +ad_{b}^{p-1}(eu)
\end{aligned}
$$
where the last equation follows from the fact that $u^p=0$ in 
$H/C_{+} H$. Similarly to \eqref{E8.5.2}, one checks that
$$\begin{aligned}
ad_b(u)&=(x\otimes 1)[(x\otimes 1+1\otimes x)^{p-1}-x^{p-1}\otimes 1]\\
&\quad +(1\otimes x)[(x\otimes 1+1\otimes x)^{p-1}-1\otimes x^{p-1}]\\
&=(x\otimes 1)(x\otimes 1+1\otimes x)^{p-1}-x^{p}\otimes 1\\
&\quad +(1\otimes x)(x\otimes 1+1\otimes x)^{p-1}-1\otimes x^{p}\\
&=(x\otimes 1+1\otimes x)^{p}-x^{p}\otimes 1-1\otimes x^{p}\\
&=x^{p}\otimes 1+1\otimes x^{p}-x^{p}\otimes 1-1\otimes x^{p}\\
&=0.
\end{aligned}
$$
Therefore $\Delta(z)=z\otimes 1+ 1\otimes z -e u$. This finishes the 
proof.
\end{proof}

\begin{remarks}
\label{xxrem8.11}
(1) In the light of Propositions \ref{xxpro8.9}-\ref{xxpro8.10} 
and the known situation for enveloping algebras of finite 
dimensional Lie algebras over $\Bbbk$, we conjecture that 
the following question has a positive answer:

\begin{question}
\label{xxque8.12}
Let $\Bbbk$ be of positive characteristic $p$.
Does every $n$-step $\Bbbk$-IHOE $H$ have a Hopf center
$C(H)$ over which $H$ is a finite module?
\end{question}

\noindent (2) For an algebraically closed field $\Bbbk$ 
of positive characteristic $p$, the connected Hopf 
$\Bbbk$-algebras of dimension $p^n$ for $p > 2$ and 
$n \leq 3$ have been classified in a series of papers 
by Nguyen, L. Wang and X. Wang \cite{NWW1, Wa1, Wa2}, 
culminating in \cite{NWW2}. The Hopf algebra 
$\overline{H}$ in Proposition \ref{xxpro8.9}(2) is 
the one in \cite[Theorem 7.4(1)]{Wa1}, and the Hopf 
algebra $\overline{H}$ in Proposition \ref{xxpro8.9}(3) 
is isomorphic to the one in \cite[Theorem 7.4(5)]{Wa1}.
The Hopf algebra $\overline{H}$ in Proposition 
\ref{xxpro8.10}(2) is isomorphic to the one 
listed in Table 5 with type T1 in the third case 
in \cite[p.858]{NWW2}. The Hopf algebra $\overline{H}$ 
in Proposition \ref{xxpro8.10}(3) can be written as 
$\Bbbk\langle X,Y,Z\rangle$ (where $X=y+z$, $Y=x$ and $Z=-b_0^{-1}d_0z$), 
with relations
$$ [X,Y]=Y, \quad [X,Z]=0, \quad [Y,Z]=0,$$
and
$$ X^p=X,\qquad  Y^p=0, \qquad Z^p=0,$$
and $X,Y$ being primitive and 
$$\Delta(Z)=Z\otimes 1+1\otimes Z+\sum_{i=1}^{p-1} \lambda_i Y^i\otimes Y^{p-i}.$$
Hence this Hopf algebra is isomorphic to the one in 
\cite[Theorem 1.3(B1)]{NWW1}. 

The following questions are thus very natural:

\begin{question}
\label{xxque8.13}
(a) Which finite dimensional connected Hopf $\Bbbk$-algebras 
can be realised as factors of IHOEs over $\Bbbk$? 

\noindent 
(b) If Question \ref{xxque8.12} has a positive answer, then 
which connected Hopf $\Bbbk$-algebras can be realised as 
the restricted Hopf algebras $H/C(H)_{+}H$ of $\Bbbk$-IHOEs $H$?
\end{question}

Regarding(a), note that a finite dimensional Hopf algebra $T$ 
which can be realised 
as a factor of an IHOE necessarily has some rather strict 
structural constraints: namely, it will have a chain of Hopf 
subalgebras $T_i$, for $i = 1, \ldots , t$ such that $T_1 = \Bbbk$ 
and $T_{i+1} = \Bbbk \langle T_i, x_{i+1}\rangle$ for elements 
$x_2, \ldots , x_n$, with a corresponding ``PBW-structure".  It 
is thus clear that not all finite dimensional 
connected Hopf $\Bbbk$-algebras can be so realised - 
for example, the restricted enveloping algebra of 
$\mathfrak{sl}(3,\Bbbk)$ presumably cannot 
be presented as a factor of an IHOE. Conversely, however, 
Xingting Wang has informed us \cite{Wa3} that he has checked 
case by case that all connected Hopf $\Bbbk$-algebras of 
dimension at most $p^3$ can be realised as Hopf factors 
of IHOEs over $\Bbbk$, at least when $p > 2$. 

\noindent (3) All of the speculations in remarks (1) and (2)  
can be reformulated with IHOEs over $\Bbbk$ replaced by 
\emph{all} connected Hopf 
$\Bbbk$-algebras of finite Gel'fand-Kirillov dimension.
\end{remarks}

%\newpage

\vspace{0.5cm}

\end{document}